\documentclass[twoside,11pt]{article}

\usepackage[english]{babel}
\usepackage{amsmath}
\usepackage{amsthm}
\usepackage{cases}
\usepackage{amsfonts}
\usepackage{amssymb}
\usepackage{graphicx}
\usepackage{colortbl,dcolumn}
\usepackage{float}
\usepackage{paralist}  
\allowdisplaybreaks[4]

\newtheorem{lemma}{\bf Lemma}[section]
\newtheorem{theorem}{\bf Theorem}[section]
\newtheorem{proposition}{\bf Proposition}[section]

\newtheorem{remark}{\bf Remark}[section]

\numberwithin{equation}{section}

\newcommand{\be}{\begin{equation}}
	\newcommand{\ce}{\end{equation}}

\newcommand\bes{\begin{eqnarray}}
	\newcommand\ees{\end{eqnarray}}
\newcommand\bess{\begin{eqnarray*}}
	\newcommand\eess{\end{eqnarray*}}

\begin{document}
\title{{\Large Nonlinear stability of shock profiles to Burgers' equation\\ with critical fast diffusion and singularity}
	\footnotetext{\small
		*Corresponding author.}
	\footnotetext{\small E-mail addresses: lixw928@nenu.edu.cn (X. Li),\ \ lijy645@nenu.edu.cn (J. Li),\\ ming.mei@mcgill.ca (M. Mei), \ \ jean-christophe.nave@mcgill.ca (J.-C. Nave).} }

\author{{Xiaowen Li$^{1,3}$, Jingyu Li$^1$, Ming Mei$^{2,3}$$^\ast$ and Jean-Christophe Nave$^3$}\\[2mm]
	\small\it $^1$School of Mathematics and Statistics, Northeast Normal University,\\
	\small\it   Changchun, 130024, P.R.China \\
	\small\it $^2$Department of Mathematics, Champlain College Saint-Lambert,\\
	\small\it     Saint-Lambert, Quebec, J4P 3P2, Canada\\
	\small\it $^3$Department of Mathematics and Statistics, McGill University,\\
	\small\it     Montreal, Quebec, H3A 2K6, Canada  }

\date{}

\maketitle

\begin{quote}
	\small \textbf{Abstract}:
	In this paper we propose the first framework to study Burgers' equation featuring critical fast diffusion in form of $u_t+f(u)_x = (\ln u)_{xx}$. The solution possesses a strong singularity when $u=0$ hence bringing technical challenges. The main purpose of this paper is to investigate the asymptotic stability of viscous shocks, particularly those with shock profiles vanishing at the far field $x=+\infty$. To overcome the singularity, we introduce some weight functions and show the nonlinear stability of shock profiles through the weighted energy method. Numerical simulations are also carried out in different cases of fast diffusion with singularity, which illustrate and confirm our theoretical results.
	
	
	\indent \textbf{AMS (2010) Subject Classification}: 35B35, 35L65, 76N10

\end{quote}


\section{Introduction}
Burgers equation is a nonlinear partial differential equation deeply rooted in fluid mechanics. It has appeared in a multitude of physical phenomena such as jet flow \cite{jet flow}, molecular interface growth \cite{growing interfaces} and traffic flow \cite{traffic flow}, etc. With the growing diversity of physical phenomena, various variants of the Burgers equation have emerged. Among them, a notable class with nonlinear diffusion, can be expressed in porous media flow:
\begin{equation}\label{porous media equ}
	u_t+f(u)_x=\left(|u|^{m-1}u \right)_{xx},
\end{equation}
where $u$ denotes the fluid velocity and $f$ is a sufficiently smooth function. $m>0$ is a parameter, which
determines whether the fluid is slow diffusion when $m>1$ or fast diffusion when $0<m<1$. If $u\geq0$, then \eqref{porous media equ} can be written by
\begin{equation}\label{burgers equ}
	u_t+f(u)_x=\left(u^{m} \right)_{xx}.
\end{equation}

When $m=1$, the equation \eqref{burgers equ} serves as a prototype for various variants of the Burgers equation and was initially considered by Burgers \cite{burgers1,burgers2} to investigate the turbulence phenomena arising from the interaction between convection and diffusion in
fluids, generating a wealth of intriguing stability results in the field. Historically, the first result on the stability of shock profiles for the Cauchy problem of \eqref{burgers equ} was due to
Ili'in and Oleinik \cite{Ili'in}. They utilized the maximum principle to prove that the solution of \eqref{burgers equ} asymptotically converges to the solution of the Riemann problem for the inviscid Burgers equation as $t\rightarrow+\infty$. Subsequently, Sattinger \cite{spectural} showed the stability based on spectral analysis. Later, Kawashima and Matsumura \cite{same line} and  Nishihara \cite{Nishihara} proved the stability as well as convergence rates by $L^2$-energy method. When the flux $f(u)$ is nonconvex, the stability was investigated by Kawashima and Matsumura \cite{nonconvex} including the system
case, and the convergence rates  by Mei \cite{M. Mei} and Matsumura and Nishihara \cite{conventional energy method}, respectively. Furthermore, Kim \cite{Kim} considered  viscous shocks and $N$-waves of classical Burgers' equation ($f(u)=u^2/2$), and the stability of viscous shock profiles. The metastability of $N$-waves have been investigated in \cite{Beck, McQuighan} using linear stability analysis and dynamical systems methods. See also the significant contributions by
Weinberger \cite{nonconvex3}, Jones \textit{et al.} \cite{Jones}, Freist\"uhler and Serre \cite{Serre}, Howard \cite{Howard-1, Howard-2,Howard-3}, Engler \cite{Engler}, and the references therein.

Regarding the system of conservation laws, the stability of shock profiles was first independently contributed by Matsumura and Nishihara \cite{Matsumura-Nishihara-1} and Goodman \cite{Goodman}, then significantly developed by
Liu \cite{Liu}, Szepessy and Xin \cite{Szepessy}, Fries \cite{Fries}, Matsumura and Mei \cite{Matsumura-Mei,Matsumura-Mei2}, Mei and Nishihara \cite{Mei-Nishihara}, Mascia and Zumbrun \cite{Mascia-Zumbrun}, Liu and Zeng \cite{Liu-Zeng-1, Liu-Zeng-2}, Kang, Vasseur and Wang \cite{Kang-Vasseur-Wang}, and the references therein.



Compared with the large number of results available for scalar viscous conservation laws \eqref{burgers equ} with $m=1$, a fundamental question such as that of the global well-posedness of \eqref{burgers equ} with $0<m<1$ or $m>1$ still remain poorly understood, and specifically for the critical fast diffusion case $(\frac{u_x}{u})_x=(\ln u)_{xx}$. The primary obstacle here is that the occurrence of singularity for  $(u^m)_{xx}=m(\frac{u_x}{u^{1-m}})_x$  as $0<m<1$ and for $(\ln u)_{xx}$ when $u=0$, and the formation of  solutions with sharp corners caused by the degeneracy for  $m>1$, respectively. These pose inherent technical difficulties preventing many classical results and conclusions. Furthermore, the nonlinear viscous term $(u^m)_{xx}$ with $m\neq1$ of equation \eqref{burgers equ}, while better aligned with the physical context, introduces complications in developing estimates compared to the linear viscous term $u_{xx}$.

In this paper, we are mainly interested in the following Burgers' equation with the critical fast diffusion:
\begin{equation}\label{original model}
	\begin{cases}
		u_t-(\ln u)_{xx}+f(u)_x=0,\quad x\in \mathbb{R},~t>0,\\
		u(x,0)=u_0(x),\quad x\in \mathbb{R},
	\end{cases}
\end{equation}
where $f$ under consideration is suitably smooth, and the initial data $u_0(x)\geq 0$ satisfies
\begin{equation}\label{original initial data1}
	u_0(x)\rightarrow u_\pm  \text{ as } x\rightarrow\pm \infty\text{ with }0=u_+<u_-.
\end{equation}

The main task of the paper is to show the asymptotic stability of viscous shocks with singular state $u_+=0$. Namely, we shall prove that:
\begin{enumerate}[(i)]
	\item the Cauchy problem \eqref{original model}-\eqref{original initial data1} admits a unique monotone shock profile $U(x-st)$ connecting $0$ and $u_-$, which satisfies
	$U_z(z)<0$ for all $z\in \mathbb{R}$ (see Theorem \ref{Existence of the shock profile});
	\item the unique shock profile $U$ obtained above is asymptotically stable. Actually, we show that if the initial value $u_0$ is a small perturbation of
	the shock profile $U$ in some precise topological sense, then the solution of \eqref{original model}-\eqref{original initial data1} will
	converge to $U$ pointwise as time tends to infinity (see Theorem \ref{stability theorem}).
\end{enumerate}

As far as the authors know, this is the first result on the global well-posedness and asymptotic dynamics of the system \eqref{original model}-\eqref{original initial data1} for fast diffusion. In contrast to the previous studies, our work takes into account two aspects. On one hand, we consider the critical case $(u_x/u)_x=(\ln u)_{xx}$ of fast diffusion, that is $m=0$ of $(\frac{u_x}{u^{1-m}})_x$. While on the other hand, we address the situation where the velocity approaches 0 as $z\rightarrow+\infty$, and successfully resolve the challenges brought by the singular nature of the problem. Below we shall briefly present the strategies used to solve our problem.

Given the particular structure of  \eqref{original model}, we can deduce that a shock profile $U$, if it exists, must meet both the Rankine-Hugoniot (R-H) condition and the generalized shock condition. On the contrary, the result (i) can be obtained by a first order differential equation (ODE) satisfied by $U$, employing the R-H condition and shock condition together with the inverse function theorem requiring only a moderate level of analytical effort. For the asymptotic stability of $U$ stated in (ii), our primary focus lies in handling the singularity $(u_x/u)_x$ afore-mentioned, which is resolved by employing the technique of taking anti-derivatives and the method of weighted energy estimates. However, all   weight functions were carefully choosen to also possess the singularities, which results in the appearance of non-zero boundary terms in the $H^1$ and $H^2$ estimates when performing integration by parts. To overcome this barrier, we introduce an approximate weight function with a parameter $\epsilon>0$ to remove the part of singularity and establish the uniform estimates that are independent of $\epsilon$. Finally, by employing Fatou's Lemma, we derive the desired estimates. This is the outline of our approach, and the precise procedures will be presented in Section 4 below.

Although we consider the system \eqref{burgers equ} for the critical fast diffusion, the ideas developed in this paper may be applicable to the normal fast diffusion ($0<m<1$) \cite{Xu-Mei-Qin-Sheng}. However, one has to face new difficulties for the stability of shock profiles, as the parameter $m$ emerges, even in the presence of reduced singularity. In order to attain a desired estimate, it may be necessary to impose further constraints on the smooth function $f$ and carefully select an appropriate weight function. The approach based on the selection of weight functions is flexible and is expected to be useful for the study of $m>1$.

Before concluding this section, we mention some other works comparable to the current work. Regarding problem \eqref{burgers equ} with $m=1$ with the boundary condition, namely, the initial-boundary value problem for the scalar viscous conservation laws, the asymptotic behavior of the solutions was first discussed by Liu and Yu \cite{IBVP1}. Then, it was also discussed in \cite{IBVP2, IBVP3}  on the half-line and bounded interval, respectively. In \cite{Hashimoto} Hashimoto is concerned with stability of the stationary solution of  Burgers' equation in exterior domains on multidimensional spaces. When $f(u)_x$ of \eqref{burgers equ} is replaced by $u^p$, Galaktionov \cite{Galaktionv} proved that the critical blow-up index is $m+2$ with slow diffusion, that is, the solution of the porous media equation blows up in finite time for any non-negative and non-trivial initial value when $1<p<m+2$. Subsequently, Qi \cite{Qi} and Mochizuki \cite{Mochizuki} extended the above results to the fast diffusion and the critical extinction index is $m$ pointed out by Li \cite{Y. Li}. In comparison to solutions blow up, the extinction refers to the phenomenon where the solution becomes zero after a finite time. Furthermore, there have also been studies on the blow-up \cite{Galaktionv2, Qi2} and extinction properties \cite{Tian, Yin1} of solutions for the Cauchy-problem to the $p$-Laplace equation, which were later expanded to more general non-Newtonian polytropic filtration equations \cite{Jin, Yin2}.

The rest of this paper is arranged as follows. In Section 2, we state our main results on the existence and nonlinear stability of shock wave solutions to the viscous conservation laws \eqref{original model}. In Section 3,
the existence of shock profiles will be discussed. In Section 4, we show the details of weighted energy estimates and prove the nonlinear stability results with $u_+=0$.
In Section 5, we carry out some numerical simulations in different cases for the fast diffusion with singularity to illustrate and confirm the viscous shock waves behavior.

\textbf{Notations}

Before proceeding, we introduce some notations for convenience. In this paper, the symbol $C$ represents a generic positive constant that may vary from one line to another. The integrals $\int_{\mathbb{R}}f(x)\mathrm{d}x$ and $\int_{0}^{t}\int_{\mathbb{R}}f(x,\tau)\mathrm{d}x\mathrm{d}\tau$ will be abbreviated as $\int f(x)$ and $\int_{0}^{t}\int f(x,\tau)$, respectively, for convenience, if no confusion arises. $H^{k}(\mathbb{R})$ is the usual $k$-th order Sobolev space defined on $\mathbb{R}$ with norm $\|f\|_{H^k(\mathbb{R})}:=\left(\sum_{j=0}^k\|\partial_x^j f\|_{L^2(\mathbb{R})}^2\right)^{1 / 2}$. $H^{k}_{w}(\mathbb{R})$ denotes the weighted space of measurable functions $f$ such that $\sqrt{w} \partial_x^j f \in L^2$ for $0\leq j\leq k$ with norm $\|f\|_{H_w^k(\mathbb{R})}:=\left(\sum_{j=0}^k \int w(x)|\partial_x^j f|^2 d x\right)^{1 / 2}$. For simplicity, we denote $\|\cdot\|:=\|\cdot\|_{L^2(\mathbb{R})}$, $\|\cdot\|_w:=\|\cdot\|_{L_w^2(\mathbb{R})}$, $\|\cdot\|_k:=\|\cdot\|_{H^k(\mathbb{R})}$ and $\|\cdot\|_{k, w}:=\|\cdot\|_{H_w^k(\mathbb{R})}$.

\section{Preliminaries and main results}		
Let $u_\pm$ be the state constants such that $0=u_+<u_-$, and $s$ be the velocity of shocks. The shock profile of \eqref{original model} with \eqref{original initial data1} is a non-constant smooth solution $u(x,t)=U(x-st)$ satisfying
\begin{equation}\label{shock profile equ}
	\begin{cases}
		-sU_{z}-\left(\frac{U_{z}}{U}\right)_{z}+f(U)_{z}=0,\quad z\in \mathbb{R}, \\
		U(+\infty)=0,~ U(-\infty)=u_{-},
	\end{cases}
\end{equation}
with $z=x-st$, which connects $0$ and $u_-$.

We first state the sufficient and necessary conditions for the existence of the shock profile $U(x-st)$ to system \eqref{original model} as follows.
	
\begin{theorem}[Existence of shock profiles] \label{Existence of the shock profile}
Let
\begin{equation}\label{g1}
	g(u)\triangleq f(u)-f(u_{\pm})-s(u-u_{ \pm}).
\end{equation}
Suppose $f\in C^{\max\{k_{\pm}+1, 3\}}(\mathbb{R})$, $k_{\pm}\geq 0$, and $g'(u_\pm)=\cdots=g^{(k_{\pm})}(u_\pm)=0$ with $g^{(k_{\pm}+1)}(u_\pm)\neq0$. \begin{enumerate}[(i)]
\item If \eqref{original model} admits a shock profile $U(x-st)$ connecting $0$ and $u_-$, then $u_-$ and $s$ must satisfy the Rankine-Hugoniot condition
\begin{equation}\label{R-H}
	 s=\frac{f(0)-f(u_-)}{0-u_-},
\end{equation}
and the generalized shock condition
\begin{equation}\label{Lax's}
	 g(u)<0, \quad\text{for } u\in(0,u_-).
\end{equation}
\item Conversely, suppose that \eqref{R-H} and \eqref{Lax's} hold. Then there exists a shock profile $U(x-st)$ of \eqref{original model} connecting $0$ and $u_-$. The shock profile $U(z)$ is unique up to a shift in $z$ and satisfies
\begin{equation}\label{U_z<0}
	 U_z(z)<0,\quad \forall z\in \mathbb{R}.
\end{equation}
Moreover, it holds as $z\rightarrow\pm\infty$,
\begin{equation}\label{f''>0}
	 \left|U(z)-0\right|\sim|z|^{-1},~\left|U(z)-u_-\right|\sim\mathrm{e}^{-\lambda_-|z|},\quad\text{if }f'(0)<s<f'(u_-),
\end{equation}
\begin{equation}\label{f'(u_+)=s}
\left|U(z)-0\right|\sim|z|^{-\frac{1}{1+k_{+}}},	 \quad\text{if }f'(0)=s,
\end{equation}
and
\begin{equation}\label{f'(u_-)=s}
	\left|U(z)-u_-\right|\sim|z|^{-\frac{1}{k_{-}}},	 \quad\text{if }f'(u_-)=s,
\end{equation}
with $\lambda_-=u_{-}(f'(u_-)-s)$.
\end{enumerate}
\end{theorem}
\begin{remark}
	 It is noted that the generalized shock condition \eqref{Lax's} implies the degenerate Lax's entropy condition
	 \begin{equation}\label{degenerate lax's}
	 	 f'(0)\leq s \leq f'(u_-).
	 \end{equation}
In fact, when we fix $u_+$ of \eqref{g1}, we have
\begin{equation}\nonumber
	\begin{aligned}
		g(U)&=f(U)-f(0)-s(U-0)\\&=\left(\frac{f(U)-f(0)}{U-0} -s\right)(U-0),		
	\end{aligned}
\end{equation}
which, in combination with $g(U)<0$ and $U>0$ for $U\in(0,u_-)$, yields
\begin{equation}\label{equ2}
	\frac{f(U)-f(0)}{U-0} -s<0.
\end{equation}
Taking the limit of \eqref{equ2} as $U$ tends to $0$, we then arrive at
\begin{equation}\label{equ3}
	f'(0)\leq s.
\end{equation}
Applying the same procedure for $u_-$ in \eqref{g1} and combining with \eqref{equ3}, one obtains \eqref{degenerate lax's}.
\end{remark}
	
The degenerate Lax's entropy condition \eqref{degenerate lax's} includes the following three cases: the nondegenerate shock condition
\begin{equation}\label{nondegenerate lax's}
	f'(0)<s<f'(u_-),
\end{equation}
the degenerate shock condition
\begin{equation}\label{degenerate lax's 1}
	f'(0)=s<f'(u_-),
\end{equation}
and
\begin{equation}\label{degenerate lax's 2}
	f'(0)<s=f'(u_-).
\end{equation}
We therefore categorize our discussion on stability into three different situations.
	
Given a shock profile $U(z)$, we expect $U(z+x_0)$ to be the asymptotic profile of the original solution $u(x,t)$ to the Cauchy problem \eqref{original model} with given initial data $u_0(x)$. We may
determine the shift $x_0$ from the perturbed equation around $U(x-st+x_0)$:
\begin{equation}\nonumber
\left(u-U\right)_t+\left(f(u)-f(U) \right)_x-\left(\frac{u_x}{u}-\frac{U_x}{U}\right)_x=0.	
\end{equation}
Integrating the above equation with respect to $x$ over $(-\infty,+\infty)$, we formally have
\begin{equation}\nonumber
	 \frac{\mathrm{d}}{\mathrm{d}t}\int _{-\infty}^{+\infty} \left(u(x,t)-U(x-st+x_0)\right)\mathrm{d} x=0,
\end{equation}
which implies
\begin{equation}\label{anti-derivative}
\int _{-\infty}^{+\infty}\left(u(x,t)-U(x-st+x_0)\right)d x=\int_{-\infty}^{+\infty} \left(u_0(x)-U(x+x_0)\right)\mathrm{d} x.	
\end{equation}
In order to set the work space in $H^2(\mathbb{R})$ for the anti-derivative of $u-U$, we thus heuristically expect
\begin{equation}\label{anti-derivative2}
\int_{-\infty}^{+\infty} \left(u_0(x)-U(x+x_0)\right)\mathrm{d} x=0.	
\end{equation}
Namely,
\begin{equation}\nonumber
	\begin{aligned}
		0&=\int_{-\infty}^{+\infty}\left(u_0(x)-U\left(x+x_0\right) \right)\mathrm{d} x \\
		& =\int_{-\infty}^{+\infty}\left(u_0(x)-U(x)\right) \mathrm{d} x-\int_{-\infty}^{+\infty}\left(U\left(x+x_0\right)-U(x)\right)\mathrm{d} x\\
		&=\int_{-\infty}^{+\infty}\left(u_0(x)-U(x)\right) \mathrm{d} x-\int_{-\infty}^{+\infty}\int_{0}^{x_0} U'(x+\eta)\mathrm{d}\eta\mathrm{d} x\\
		&=\int_{-\infty}^{+\infty}\left(u_0(x)-U(x)\right) \mathrm{d} x-\int_{0}^{x_0}\int_{-\infty}^{+\infty} U'(x+\eta)\mathrm{d} x\mathrm{d}\eta\\
		&=\int_{-\infty}^{+\infty}\left(u_0(x)-U(x)\right) \mathrm{d} x-\int_{0}^{x_0}\left(U(+\infty)-U(-\infty) \right) \mathrm{d}\eta\\
		&=\int_{-\infty}^{+\infty}\left(u_0(x)-U(x)\right) \mathrm{d} x-\int_{0}^{x_0}\left(0-u_-\right) \mathrm{d}\eta\\
		& =\int_{-\infty}^{+\infty}\left( u_0(x)-U(x)\right) \mathrm{d} x+x_0u_{-}.
	\end{aligned}
\end{equation}
This gives
\begin{equation}\label{shift}
	 x_0:=-\frac{1}{u_-}\int_{-\infty}^{+\infty}\left(u_0(x)-U(x)\right) \mathrm{d}x.
\end{equation}
Let us define
\begin{equation}\label{phi0}
	\phi_0(x)=\int_{-\infty}^{x}\left(u_0(y)-U(y+x_0) \right)\mathrm{d}y,
\end{equation}
and thus $\phi_0(+\infty)=0$. Our stability of shock profile $U(z+x_0)$ with singularity at $u_+=0$ is stated as follows.	
\begin{theorem}[Stability of shock profiles] \label{stability theorem}
Let $U(z)$ be a shock profile obtained in Theorem \ref{Existence of the shock profile}. Assume that $u_0(x)-U(x)\in L^1(\mathbb{R})$. Define $x_0$ and $\phi_0(x)$ by \eqref{shift} and \eqref{phi0}. Then it holds:
\begin{enumerate}[(i)]
\item When $f'(0)<s<f'(u_-)$, then there exists a positive constant $\epsilon_{1}$ such that if $\left\|\phi_0\right\|+\left\|u_0-U\right\|_{1, w_1} \leq \epsilon_1$ with
\begin{equation}
	 w_1(U)=U^{-2}\sim\left\{\begin{array}{cc}
	 	 1+z^2,& z > 0, \\
	 	1, & z\leq0,
	 \end{array}\right.
\end{equation}
the Cauchy problem \eqref{original model}-\eqref{original initial data1} has a unique global solution $u(x,t)$ satisfying
\begin{equation}\nonumber
	u-U \in C\left([0, \infty) ; H_{w_1}^1\right) \cap L^2\left((0, \infty) ; L_{w_2}^2\right),\quad (u-U)_x \in L^2\left((0, \infty) ; H_{w_3}^1\right),
\end{equation}
where $U=U(x-st+x_0)$ is the shifted shock profile, and the weight functions $w_i(U)$, $i=2, 3$, are defined by
\begin{equation}\label{weight function}
 w_2(U)=U^{-1}\sim\left\{\begin{array}{cc}
 	\sqrt{1+z^2},& z > 0, \\
 	1, & z\leq0,
 \end{array}\right.\quad w_3(U)=U^{-3}\sim \left\{\begin{array}{cc}
 1+z^3,& z > 0, \\
 1, & z\leq0,
\end{array}\right.	
\end{equation}
and moreover
\begin{equation}
	\sup _{x \in \mathbb{R}}|u(x, t)-U(x-s t+x_0)| \rightarrow 0, \quad \text { as } t \rightarrow+\infty.
\end{equation}
\item When $f'(0)=s<f'(u_-)$, then there exists a positive constant $\epsilon_{2}$ such that if $\left\|\phi_0\right\|_{w_4}+\left\|u_0-U\right\|_{1, w_1} \leq \epsilon_2$ with
\begin{equation}\label{2.22}
	 w_4(U)=\frac{U\left(U-u_-\right)}{g(U)}\sim \left\{\begin{array}{cc}
	 	\left( 1+z^{k_+}\right)^\frac{1}{1+k_+},& z > 0, \\
	 	1, & z\leq0,
	 \end{array}\right.
\end{equation}
the Cauchy problem \eqref{original model}-\eqref{original initial data1} has a unique global solution $u(x,t)$ satisfying
\begin{equation}\nonumber
	u-U \in C\left([0, \infty) ; H_{w_1}^1\right) \cap L^2\left((0, \infty) ; L_{w_5}^2\right),\quad (u-U)_x \in L^2\left((0, \infty) ; H_{w_3}^1\right),
\end{equation}
where the weight function $w_5(U)$ is defined by
\begin{equation}\label{2.23}
	w_5(U)=\frac{\left(U-u_-\right)}{g(U)}\sim \left\{\begin{array}{cc}
		\sqrt{1+z^2},& z > 0, \\
		1, & z\leq0,
	\end{array}\right.
\end{equation}
and moreover
\begin{equation}
	\sup _{x \in \mathbb{R}}|u(x, t)-U(x-s t+x_0)| \rightarrow 0, \quad \text { as } t \rightarrow+\infty.
\end{equation}
\item When $f'(0)<s=f'(u_-)$, then there exists a positive constant $\epsilon_{3}$ such that if $\left\|\phi_0\right\|_{w_4}+\left\|u_0-U\right\|_{1, w_1} \leq \epsilon_3$ with
\begin{equation}\label{2.25}
w_4(U)=\frac{U\left(U-u_-\right)}{g(U)}\sim\left\{\begin{array}{cc}
	\sqrt{1+z^2},& z < 0, \\
	1, & z\geq0,
\end{array}\right.	
\end{equation}
the Cauchy problem \eqref{original model}-\eqref{original initial data1} has a unique global solution $u(x,t)$ satisfying
\begin{equation}\nonumber
	u-U \in C\left([0, \infty) ; H_{w_1}^1\right) \cap L^2\left((0, \infty) ; L_{w_5}^2\right),\quad (u-U)_x \in L^2\left((0, \infty) ; H_{w_3}^1\right),
\end{equation}
where the weight function $w_5(U)$ is defined by
\begin{equation}\label{weight function2}
	w_5(U)=\frac{\left(U-u_-\right)}{g(U)}\sim \sqrt{1+z^2},\quad z\in\mathbb{R}.		
\end{equation}
and moreover
\begin{equation}
	\sup _{x \in \mathbb{R}}|u(x, t)-U(x-s t+x_0)| \rightarrow 0, \quad \text { as } t \rightarrow+\infty.
\end{equation}
\end{enumerate}
\end{theorem}

\section{Existence of the shock profile}	
This section is devoted to showing that there exists a shock profile
$U(x-st)$ connecting $0$ and $u_-$, if and only if $u_-$ and $s$ satisfy the Rankine-Hugoniot condition \eqref{R-H} and the generalized shock condition \eqref{Lax's}; along this line, we prove Theorem \ref{Existence of the shock profile}.
	
\begin{proof}[Proof of Theorem \ref{Existence of the shock profile}]	
To prove $(i)$, we suppose that \eqref{original model} admits a shock profile $U(x - st)$
connecting $0$ and $u_-$. Therefore, $U(z)$ with $z=x-st$, must satisfy \eqref{shock profile equ} and $U_z(z)<0$ due to $u_->0$. Formally, we expect that $\lim\limits_{z\rightarrow\pm\infty}\frac{U_z}{U}(z)=0$, then integrating \eqref{shock profile equ} in $z$ over $\mathbb{R}$ yields
\begin{equation}\nonumber
	-s\left(0-u_{-}\right)+f\left(0\right)-f\left(u_{-}\right)=0,
\end{equation}
which implies the Rankine-Hugoniot condition \eqref{R-H}. Additionally, integrating \eqref{shock profile equ} in $z$ over $(z,+\infty)$, we have
\begin{equation}\label{eq1}
	-s U-\frac{U_z}{U}+f(U)=a,
\end{equation}
where $a$ is an integral constant. Since $U( \pm \infty)=u_{ \pm}$ and $\lim\limits_{z\rightarrow\pm\infty}\frac{U_z}{U}(z)=0$, letting $z\rightarrow\pm\infty$ in \eqref{eq1}, we see easily that
$$
a=-s u_{ \pm}+f\left(u_{ \pm}\right).
$$
This along with \eqref{eq1} gives
\begin{equation}\label{ODE}
	 U_z=U\left[f(U)-f(u_{ \pm})-s(U-u_{ \pm})\right].
\end{equation}
Define
\begin{equation}\label{g}
	g(U)\triangleq f(U)-f(u_{ \pm})-s(U-u_{ \pm}),
\end{equation}
and let
\begin{equation}\label{h(U)}
	h(U)\triangleq U\left[f(U)-f(u_{ \pm})-s(U-u_{ \pm})\right].
\end{equation}
Thus, $h(U)=Ug(U)$. From standard ODE theory, the equation \eqref{ODE} with $h(u_\pm)=0$ admits a smooth solution $U(z)$ satisfying $U(\pm\infty)=u_\pm$ if and only if
\begin{equation}
	 h(U)<0,\quad\forall U\in(0,u_-).
\end{equation}
Here and hereafter $(0,u_-)$ denotes the set of values $u$ or $U$ between $0$ and $u_-$ ($0$ and $u_-$ are not included). Consequently, the generalized shock condition \eqref{Lax's} holds due to $U>0$ for $U\in(0,u_-)$. This completes the proof of $(i)$.

Conversely, we suppose that \eqref{R-H} and \eqref{Lax's} hold. By virtue of \eqref{Lax's}, \eqref{ODE} and \eqref{h(U)}, we only need to find the global solution of the following ODE
\begin{equation}\label{eq4}
	\frac{\mathrm{d}z}{\mathrm{d}U}=\frac{1}{h(U)}.
\end{equation}
In this case, we can solve \eqref{eq4} in the form
\begin{equation}\nonumber
	 z+\text{constant}=\int_{u_\star}^{U}\frac{1}{h(y)}\mathrm{d}y\triangleq H(U),
\end{equation}
where $u_\star=\left(u_++u_-\right)/2$. Since for any given $U\in(0,u_-)$, the function $\frac{1}{h(y)}$ is integrable over $(u_\star,U(z))$ and $H'(U)=\frac{1}{h(U)}<0$ due to $h(U)=Ug(U)$ and \eqref{Lax's}, we know that $\int_{u_\star}^{U}\frac{1}{h(y)}\mathrm{d}y$ is finite and monotonically decreasing for any $U\in(0,u_-)$. It remains to check the limiting values $U(\pm\infty)=u_{\pm}$. When $f'(0)<s<f'(u_-)$, we have $\left|g(U)\right|\sim\left|U-u_\pm\right|$ as $U\rightarrow u_\pm$ because of $g'(u_\pm)\neq0$. Therefore, $h(U)$ is equivalent to $-U^2$ when $U\rightarrow 0$ and that
\begin{equation}\label{eq18}
	\int_{u_\star}^{0}\frac{1}{h(y)}\mathrm{d}y=+\infty.	
\end{equation}
As $U\rightarrow u_-$, it becomes apparent that
\begin{equation}\label{eq19}
	\int_{u_\star}^{u_-}\frac{1}{h(y)}\mathrm{d}y=-\infty,	
\end{equation}
because of $h(U)\sim U-u_-$. When $s=f'(0)$ or $s=f'(u_-)$, since $g'(u_\pm)=\cdots=g^{(k_{\pm})}(u_\pm)=0$ with $g^{(k_{\pm}+1)}(u_\pm)\neq0$, we have
$\left|g(U)\right|\sim\left|U-u_\pm\right|^{1+k_\pm}$ for $k_\pm\geq1$ as $U\rightarrow u_\pm$ and thus \eqref{eq18} and \eqref{eq19} are satisfied. Therefore, according to the inverse function theorem, there exists a unique continuous function $H^{-1}$ such that
\begin{equation}\label{fanhanshu}
	U=H^{-1}(z-z_\star),\quad \forall z\in(-\infty,+\infty).	
\end{equation}
Thus, the existence of the shock profile $U(z)$ is proved. By virtue of \eqref{Lax's} and \eqref{ODE}, then \eqref{U_z<0} holds.

Next we are going to show the convergence rate for $U(z)$ as $z\rightarrow\pm\infty$. When $f'(0)<s<f'(u_-)$, following the asymptotic theory of ODEs, we get from \eqref{h(U)} as $z\rightarrow+\infty$ that
\begin{equation}\label{eq5}
	U_z\sim U^2\left(f'(0)-s\right),
\end{equation}
A direct calculation gives
\begin{equation}\nonumber
U\sim\left[C_1-(f'(0)-s)z \right]^{-1}~ \text{for } z>0.	
\end{equation}
Here $C_1>0$ is a bounded constant due to $U(+\infty)=0$, which implies the convergence rate $|z|^{-1}$ for $U(z)\rightarrow 0$ as $z\rightarrow+\infty$. When $z\rightarrow-\infty$, we have
\begin{equation}\label{eq6}
U_z\sim u_-\left(f'(u_-)-s\right) (U-u_-),	
\end{equation}
and hence
$$
u_{-}-U\sim C_2\mathrm{e}^{u_{-}(f'(u_-)-s)z}~ \text{for } z<0,
$$
where $C_2$ is a positive constant due to $U<u_-$ for $U\in(0,u_-)$. This implies the convergence rate for $U(z)\rightarrow u_-$ as $z\rightarrow -\infty$ is $\mathrm{e}^{-\lambda_-|z|}$ with $\lambda_-=u_{-}(f'(u_-)-s)$. In particular, when $s=f'(0)$ or $s=f'(u_-)$, we have
$\left|g(U)\right|\sim\left|U-u_\pm\right|^{1+k_\pm}$ for $k_\pm\geq1$ as $z\rightarrow \pm\infty$ as discussed above and this implies the convergence rate $|z|^{-\frac{1}{1+k_+}}$ and $|z|^{-\frac{1}{k_-}}$ for $U(z)\rightarrow u_\pm$ as $z\rightarrow \pm\infty$. Thus, the proof of Theorem \ref{Existence of the shock profile} is complete.
\end{proof}	

\begin{remark}
	 By virtue of $u_->0$, if $f'(0)<s<f'(u_-)$, it is easy to see from \eqref{eq5} and \eqref{eq6} that
	 \begin{equation}\label{|Uz|estimate}
	 	|U_z(z)| \leq CU^{2}(z)\quad \text{for all } z\in(-\infty,+\infty).
	 \end{equation}
Recalling that $U_z=Ug(U)$ and $\left|g(U)\right|\sim\left|U-u_\pm\right|^{1+k_\pm}$ for $k_\pm\geq1$ as $z\rightarrow \pm\infty$ if $s=f'(0)$ or $s=f'(u_-)$, we thus arrive at
	 \begin{equation}\label{|Uz|estimate1}
	|U_z(z)| \leq C\left|U(z)-0\right|^{2+k_+}\leq CU^{2}(z),\quad \text{as } z\rightarrow +\infty,
\end{equation}
and
\begin{equation}\label{|Uz|estimate2}
	|U_z(z)| \leq Cu_-\left|U(z)-u_-\right|^{1+k_-},\quad \text{as } z\rightarrow -\infty.
\end{equation}
\end{remark}

\section{Nonlinear stability}
In this section, we prove Theorem \ref{stability theorem} and hence establish the nonlinear stability of the shock profile of \eqref{original model}-\eqref{original initial data1}. One may observe that the system \eqref{original model} exhibits a singularity in the vicinity of $z=+\infty$, hence we shall devise new strategies to address this challenge. Before proceeding, we first use the technique of anti-derivative to reformulate the problem.
\subsection{Reformulation of the problem }
Let $U(x-st)$ be the shock profile obtained in Theorem \ref{Existence of the shock profile}. According to \eqref{anti-derivative} and \eqref{anti-derivative2}, we have
\begin{equation}\label{shift2}
	\int_{-\infty}^{+\infty}\left( u(x, t)-U\left(x-s t+x_0\right)\right) d x =\int_{-\infty}^{+\infty}\left( u_0(x)-U\left(x+x_0\right)\right)d x
		=0.
\end{equation}
We thus decompose the solution of \eqref{original model} as
\begin{equation}\label{decompose}
	u(x, t)=U\left(x-s t+x_0\right)+\phi_z(z, t),	
\end{equation}
where $z=x-st$. That is
\begin{equation}\label{u-U}
	\phi(z, t)=\int_{-\infty}^z\left(u(y, t)-U\left(y-s t+x_0\right) \right)d y,	
\end{equation}
for all $z\in \mathbb{R}$ and $t\geq0$. It then follows from \eqref{shift2} that
\begin{equation}\label{eq20}
\phi(\pm \infty,t)=0,\quad \text{for all }t>0.	
\end{equation}
Without loss of generality, we assume that the translation $x_0=0$. Substituting \eqref{decompose} into \eqref{original model}, using \eqref{shock profile equ} and \eqref{eq20} and integrating the system with respect to $z$, the problem \eqref{original model} is reduced to
\begin{equation}\nonumber
	\phi_t=s \phi_z+\left(\ln \left(U+\phi_z\right)-\ln U\right)_z-\left(f\left(U+\phi_z\right)-f(U)\right),
\end{equation}
which is rewritten as
\begin{equation}\label{phiequ}
	\phi_t+g'(U) \phi_z-\left(\frac{\phi_z}{U}\right)_z=F+G_z,
\end{equation}
where
\begin{equation}\label{F}
	 F \equiv-\left(f\left(U+\phi_z\right)-f(U)-f^{\prime}(U) \phi_z\right),
\end{equation}
and
\begin{equation}\label{G}
	G \equiv \ln \left(U+\phi_z\right)-\ln U-\frac{\phi_z}{U}.
\end{equation}
The initial perturbation of $\phi$ is thus given by
\begin{equation}\label{phi00}
	\phi(z,0)=\int_{-\infty}^z\left( u_0(y)-U(y)\right) d y~\text{with } \phi(\pm \infty,0)=0.
\end{equation}

We search for solutions of the system \eqref{phiequ} in the following space
\begin{equation}\nonumber
	\begin{gathered}
		X(0, T):=\left\{\phi(z, t) \mid \phi \in C\left([0, T] ; L^2\cap L_{w_4}^2 \right), \phi_z \in C\left([0, T] ; H_{w_1}^1\right) \cap L^2\left((0, T) ; L_{w_2}^2\cap L_{w_5}^2\right),\right. \\
		\left. \phi_{z z} \in L^2\left((0, T) ; H_{w_3}^1\right)\right\},
	\end{gathered}
\end{equation}
where the weight functions $w_i$, $i=1,\cdots,5$, are given by
\begin{equation}\label{w}
	\begin{aligned}
		&w_1(U)=U^{-2},\quad w_2(U)=U^{-1},\quad w_3(U)=U^{-3},\\&
		w_4(U)=\frac{U\left(U-u_-\right)}{g(U)},\quad w_5(U)=\frac{\left(U-u_-\right)}{g(U)}.
	\end{aligned}
\end{equation}
Define
\begin{equation}\label{N(t)}
	N(t):=\sup_{\tau\in [0,t]}\left(\left\|\phi(\cdot,\tau)\right\|+\left\|\phi_z(\cdot,\tau)\right\|_{1,w_1}\right).
\end{equation}
Clearly, if $\phi_z\in L_{w_1}^2$, then $\phi_z\in L^2$ since $w_1(U)\geq1$. Then, by the weighted Sobolev embedding inequality, it holds that
\begin{equation}\label{sobolev embedding}
	\sup_{\tau\in [0,t]}\left\{\left\|\phi(\cdot,\tau)\right\|_{L^{\infty}},\left\|\sqrt{w_1}\phi_z(\cdot,\tau)\right\|_{L^{\infty}}\right\} \leq CN(t).
\end{equation}

Theorem \ref{stability theorem} is a consequence of the following theorem.
\begin{theorem}\label{phi stability}
	Suppose $\phi_0\in L^2(\mathbb{R})\cap L_{w_4}^2(\mathbb{R})$ and that $\phi_{0z}\in H_{w_1}^{1}(\mathbb{R})$. Then there exists a positive constant $\delta_{1}$ such that if $N(0)\leq \delta_{1}$, the Cauchy problem \eqref{phiequ}-\eqref{G} with \eqref{phi00} has a unique global solution $\phi\in X(0,\infty)$ satisfying
	\begin{equation}\label{priori estimate}
		\begin{aligned}
			 \|\phi\|^2+\|\phi\|_{w_4}^2&+\left\|\phi_z\right\|_{1, w_1}^2+\int_0^t\left(\left\|\phi_z(\tau)\right\|_{w_2}^2+\left\|\phi_z(\tau)\right\|_{w_5}^2+\left\|\phi_{zz}(\tau)\right\|_{1, w_3}^2\right) \\&
		 \leq C\left(\|\phi_0\|^2+\|\phi_0\|_{w_4}^2+\left\|\phi_{0z}\right\|_{1, w_1}^2\right)\leq CN^2(0),
		\end{aligned}
	\end{equation}
	for any $t\in [0,\infty)$, where $w_i$, $i=1,\cdots,5$, are defined by \eqref{w}. Moreover, $\phi_z$ tends to $0$ in the maximum norm as $t \rightarrow+\infty$, that is
	\begin{equation}\label{asymptotic}
		\sup _{x \in \mathbb{R}}|\phi_{z}(z, t)| \rightarrow 0, \quad \text { as } t \rightarrow+\infty.
	\end{equation}	
\end{theorem}
The global existence of $\phi$, as stated in Theorem \ref{phi stability}, can be treated by the weighted energy method based on local existence with the \textit{a priori} estimates provided below \cite{same line}.

\begin{proposition}[Local existence]\label{local existence}
Suppose $\phi_0\in L^2(\mathbb{R})\cap L_{w_4}^2(\mathbb{R})$ and that $\phi_{0z}\in H_{w_1}^{1}(\mathbb{R})$. For any $\delta_0>0$, there exists a positive constant $T_0$ depending on $\delta_0$ such that if $N(0)\leq \delta_0$, then the problem \eqref{phiequ}-\eqref{G} with \eqref{phi00} has a unique solution $\phi\in X(0,T_0)$ satisfying $N(t)\leq 2N(0)$ for any $0\leq t\leq T_0$.
\end{proposition}

\begin{proposition}[{\it A priori} estimates]\label{proposition priori estimate}
	Let $\phi$ be a solution in $X(0,T)$ obtained in Proposition \ref{local existence} for a positive constant $T$. Then there exists a positive constant $\delta_{2}$, independent of $T$, such that if
	\begin{equation}\label{priori assumption}
		N(t)\leq \delta_{2}~\text{for all } 0\leq t\leq T,	
	\end{equation}
	then the estimate \eqref{priori estimate} holds for $t\in [0,T]$.
\end{proposition}

Proposition \ref{local existence} can be proved in the standard way (see \cite{local existence} for instance). So we omit its proof. However, it is crucial to establish
the \textit{a priori} estimates as stated in Proposition \ref{proposition priori estimate}, which will be proved in the next subsection.

\subsection{Basic estimates and stability theorem}	
We first derive the basic estimates, which play the key role in our proof.

\begin{lemma}
	Under the a priori assumption \eqref{priori assumption}, if $N(T)\ll1$, then it holds that
	\begin{equation}\label{F1}
		|F|\leq C \phi_{z}^2,\quad |G|\leq C \frac{\phi_{z}^2}{U^2},
	\end{equation}	
\begin{equation}\label{Fz}
|F_z|\leq C\left( |U_z|\phi_{z}^2+|\phi_{z}||\phi_{z z}|\right),
\end{equation}
\begin{equation}\label{Gz}
	|G_z|\leq C\left( \frac{|U_z|}{U^3}\phi_z^2+\frac{|\phi_{z}||\phi_{z z}| }{U^2}\right),
\end{equation}
\begin{equation}\label{Gzz}
	|G_{zz}| \leq C\left[\left(\frac{U_z^2 }{U^4}+\frac{|U_{zz}| }{U^3} \right)\phi_{z}^2+ \frac{\phi_{z z}^2 }{U^2}+\frac{|U_z| }{U^3}|\phi_{z}||\phi_{zz}|+\frac{ |\phi_{z}||\phi_{zzz}|}{U^2}\right].
\end{equation}
\begin{proof}
	From the formulations of $F$ and $G$ given in \eqref{F} and \eqref{G}, respectively, a direct calculation gives
\begin{equation}\nonumber
	\begin{aligned}
F_z&=-f^{\prime}\left(U+\phi_z\right)\left(U_z+\phi_{z z}\right)+f^{\prime}(U) U_z+f^{\prime \prime}(U) U_z \phi_z+f^{\prime}(U) \phi_{z z}	\\
&=-\left(f^{\prime}\left(U+\phi_{z}\right)-f^{\prime}(U)-f^{\prime \prime}(U) \phi_{z}\right) U_z-\left(f^{\prime}\left(U+\phi_{z }\right)-f^{\prime}(U)\right) \phi_{zz },	
	\end{aligned}
\end{equation}	
\begin{equation}\nonumber
		\begin{aligned}
		G_z&=\frac{1}{\left(U+\phi_{z}\right)}\left( U_z+\phi_{zz}\right)-\frac{U_z }{U}-\frac{ \phi_{z z}}{U}-\left( -\frac{U_z}{U^{2} }\right)\phi_{z}\\
		&=\left[\frac{ 1}{\left(U+\phi_{z}\right)}-\frac{1}{U}-\left(-\frac{1}{U^2} \right)\phi_{z} \right]U_z+\left[\frac{ 1}{\left(U+\phi_{z}\right)}-\frac{1}{U} \right]\phi_{z z},	
	\end{aligned}
\end{equation}
and
\begin{equation}\nonumber
	\begin{aligned}
G_{zz}&=	\left[\frac{\left(U_z+\phi_{zz}\right)}{U+\phi_{z}}-\frac{U_z}{U} -\frac{\phi_{z z}}{U} +\frac{U_z}{U^2} \phi_z\right]_z\\
&=\left[-\frac{1}{\left( U+\phi_z\right)^2}-\left(-\frac{1}{U^2}\right)-\frac{2\phi_z}{U^3} \right] U_z^2-\frac{\phi_{zz}^2}{\left( U+\phi_z\right)^2}+\left(\frac{1}{ U+\phi_z}-\frac{1}{U}\right) \phi_{z z z}\\&\quad+2\left[-\frac{1}{\left( U+\phi_z\right)^2}-\left(-\frac{1}{U^2}\right)\right] U_z \phi_{z z}+\left[\frac{1}{ U+\phi_z}-\frac{1}{U}-\left(-\frac{\phi_z}{U^2} \right)\right] U_{z z},
\end{aligned}
\end{equation}
where, thanks to $f\in C^{\max\{k_{\pm}+1, 3\}}(\mathbb{R})$, $\|\frac{\phi_{z}}{U}\|_{L^{\infty}}\leq CN(t)$ from \eqref{sobolev embedding} and Taylor's expansion, it is straightforward to imply \eqref{F1}, \eqref{Fz}, \eqref{Gz} and \eqref{Gzz}.
\end{proof}

\end{lemma}

\begin{lemma}\label{L2}
Let the assumptions of Proposition \ref{proposition priori estimate} hold. If $N(T)\ll1$, then there exists a constant $C>0$ independent of $T$ such that
	\begin{equation}\label{L2 estimate}
		\|\phi\|_{w_4}^2+\int_0^t\left\|\phi_z(\tau)\right\|_{w_5}^2  \leq C\|\phi_0\|_{w_4}^2,
	\end{equation}
for any $t\in[0,T]$.
\end{lemma}
\begin{proof}
	 Multiplying \eqref{phiequ} by $w_4(U)\phi(z,t)$, we have
	 \begin{equation}\label{eq7}
	 	\begin{aligned}
	 \left(\frac{1}{2}w_4(U) \phi^{2}\right)_t&+\left[\frac{1}{2}\left(gw_4 \right)^{\prime}(U) \phi^2-\frac{w_4(U)}{U} \phi \phi_z\right]_z+\frac{w_4(U)}{U} \phi_z^2\\&-\frac{1}{2}\left( gw_4\right)^{\prime \prime}(U) U_z \phi^2=w_4(U)F\phi + w_4(U)G_z\phi,		
	 	\end{aligned}
	\end{equation}
where $'=\frac{\mathrm{d}}{\mathrm{d}U}$.	Integrating \eqref{eq7} over $\mathbb{R}\times (0,t)$, yields
\begin{equation}\label{eq8}
	\frac{1}{2}\int w_4(U) \phi^{2}+\int_{0}^{t}\int\frac{w_4(U)}{U}\phi_z^2\leq\frac{1}{2}\int w_4(U)\phi_0^{2}+\int_{0}^{t}\int w_4(U)F\phi + \int_{0}^{t}\int w_4(U)G_z\phi,
\end{equation}	
where we have ignored the term $-\frac{1}{2}\int_{0}^{t}\int\left( gw_4\right)^{\prime \prime}(U) U_z \phi^2$ due to $U_z(z)<0$ for $U\in(0,u_-)$ and $\left(gw_4\right)^{\prime \prime}(U)=2$. By \eqref{F1} and $\left\|\phi\right\|_{L^{\infty}}\leq CN(t)$ of \eqref{sobolev embedding}, we deduce that
\begin{equation}\label{eq9}
	\int_{0}^{t}\int w_4(U)F\phi\leq CN(t)\int_{0}^{t}\int w_4(U)\phi_z^2\leq CN(t)\int_{0}^{t}\int\frac{w_4(U)}{U} \phi_z^2.
\end{equation}	
For the last term on the right hand side of \eqref{eq8}, after using integration by parts, it follows from \eqref{F1}, $\left\|\frac{\phi_z }{U}\right\|_{L^{\infty}}\leq CN(t)$  and $\left\|\phi\right\|_{L^{\infty}}\leq CN(t)$ of \eqref{sobolev embedding} that
\begin{equation}\label{eq10}
	\begin{aligned}
\int_{0}^{t}\int w_4(U) G_z\phi&=-\int_{0}^{t}\int G\left[ w_4(U)\phi_z+\left(w_4(U) \right)_z\phi\right]\\&\leq C\int_{0}^{t}\int \frac{w_4(U)}{U^2}|\phi_z^3|+C\int_{0}^{t}\int \frac{\left(w_4(U) \right)_z}{U^2}|\phi|\phi_z^2\\&\leq CN(t)\int_{0}^{t}\int \frac{w_4(U)}{U}\phi_z^2+CN(t)\int_{0}^{t}\int \frac{\left(w_4(U) \right)_z}{U^2}\phi_z^2,		
	\end{aligned}
\end{equation}
where the last term on the right hand side can be rewritten as
\begin{equation}\label{eq25}
CN(t)\int_{0}^{t}\int \frac{\left(w_4(U) \right)_z}{U^2}\phi_z^2=CN(t)\int_{0}^{t}\int	\frac{w_4(U)}{U}\phi_z^2 \cdot\frac{w_4^{\prime}(U)}{w_4(U)}\frac{U_z}{U}.
\end{equation}
Next we shall show
\begin{equation}\label{eq21}
	\left|\frac{w_4^{\prime}(U)}{w_4(U)} \right| \left|\frac{U_z}{U} \right|\leq C, \quad \forall z\in (-\infty,\infty).
\end{equation}
When $f'(0)<s<f'(u_-)$, since $|g(U)|\sim\left| U-u_\pm\right|$ as $U\rightarrow u_\pm$, we have
\begin{equation}\label{const}
	 w_4(U)\sim \text{constant.}
\end{equation}
This along with \eqref{|Uz|estimate} gives \eqref{eq21}. While if $f'(0)=s<f'(u_-)$, by $|g(U)|\sim\left| U-0\right|^{1+k_+}$ as $U\rightarrow 0$, we have
\begin{equation}\label{eq22}
	\begin{aligned}
\left|\frac{w_4^{\prime}(U)}{w_4(U)} \right|&=\left| \frac{\left(2U-u_-\right)g(U)-U\left(U-u_-\right)g'(U) }{U\left(U-u_-\right)g(U)}\right|\\&\leq C\frac{\left|U-0\right|^{1+k_+} \left[\left|2U-u_- \right|+|U-u_-|\right]}{\left|U-0\right|^{2+k_+}|U-u_-|}\\&\leq C\frac{1}{|U-0|}.			
	\end{aligned}
\end{equation}
Moreover, in view of \eqref{|Uz|estimate1}, one obtains
\begin{equation}\nonumber
	 \left|\frac{U_z}{U}\right|\leq C\left|U-0\right|^{1+k_+},\quad \text{as } U\rightarrow0.
\end{equation}
Therefore, combining with \eqref{eq22}, we get
\begin{equation}\label{eq23}
\left|\frac{w_4^{\prime}(U)}{w_4(U)} \right| \left|\frac{U_z}{U} \right|\leq C\left|U-0\right|^{k_+}\leq C,\quad \text{as } U\rightarrow0,
\end{equation}
where we have used the fact $k_+\geq1$ due to $g'(0)=0$. Also
if $f'(0)<s=f'(u_-)$, we have
\begin{equation}\nonumber
	\begin{aligned}
		\left|\frac{w_4^{\prime}(U)}{w_4(U)} \right|&\leq C\frac{\left|U-u_-\right|^{1+k_-} \left[\left|2U-u_- \right|+|U|\right]}{\left|U-u_-\right|^{2+k_-}|U|}\\&\leq C\frac{1}{|U-u_-|},	\quad \text{as }U\rightarrow u_-,		
	\end{aligned}
\end{equation}
which in combination with \eqref{|Uz|estimate2} implies for $k_-\geq1$ that
\begin{equation}\label{eq24}
	\left|\frac{w_4^{\prime}(U)}{w_4(U)} \right| \left|\frac{U_z}{U} \right|\leq C\left|U-u_-\right|^{k_-}\leq C,\quad \text{as }U\rightarrow u_-.	
\end{equation}
Hence, thanks to \eqref{eq23} and \eqref{eq24}, \eqref{eq21} is true for $ z\in (-\infty,\infty)$. Then, by virtue of \eqref{eq25} and \eqref{eq21}, it gives
\begin{equation}\label{eq26}
	CN(t)\int_{0}^{t}\int \frac{\left(w_4(U) \right)_z}{U^2}\phi_z^2\leq CN(t)\int_{0}^{t}\int	\frac{w_4(U)}{U}\phi_z^2.
\end{equation}
Adding \eqref{eq8}, \eqref{eq9},  \eqref{eq10} and \eqref{eq26} and taking $\delta_{2}\leq\delta_0$, we can derive \eqref{L2 estimate}, provided that $N(t)$ is suitably small such that $N(t)\leq \delta_2$. The proof of Lemma \ref{L2} is complete.
\end{proof}

\begin{remark}
	 Note that when $f'(0)<s<f'(u_-)$, $w_4(U)\sim \text{const.}$ (see \eqref{const}), we also have the estimate for $\phi$ that
\begin{equation}\label{L2 estimate1}
	\|\phi\|^2+\int_0^t\left\|\phi_z(\tau)\right\|_{w_2}^2  \leq C\|\phi_0\|^2~\text{for any }t\in[0,T].
\end{equation}

When $f'(0)=s<f'(u_-)$, since $|g(U)|\sim\left| U-0\right|^{1+k_+}$ as $U\rightarrow 0$ for $k_+\geq1$, then $w_4(U)\sim \left| U-0\right|^{-k_+}$ and $\frac{w_4(U)}{U}\sim \left| U-0\right|^{-k_+-1}$, which along with \eqref{L2 estimate} implies
\begin{equation}\label{L2 estimate2}
	\|\phi\|^2+\int_0^t\left\|\phi_z(\tau)\right\|_{w_2}^2  \leq\|\phi\|_{w_4}^2+\int_0^t\left\|\phi_z(\tau)\right\|_{w_5}^2\leq C\|\phi_0\|_{w_4}^2.
\end{equation}	
\end{remark}

Then we give the estimate of the first order derivative of $\phi$.
\begin{lemma}\label{H1}
	Let the assumptions of Proposition \ref{proposition priori estimate} hold. If $N(T)\ll1$, then it holds for any $t\in [0,T]$ that
	\begin{equation}\label{H1 estimate}
			\|\phi_z\|_{w_1}^2+\int_0^t\left\|\phi_{zz}(\tau)\right\|_{w_3}^2 \leq C\left(\|\phi_0\|^2+\|\phi_0\|_{w_4}^2+\|\phi_{0z}\|_{w_1}^2\right).
	\end{equation}	
\end{lemma}

\begin{proof}
	Denote $U_\epsilon\triangleq U+\epsilon$, where $\epsilon>0$ is a constant. Multiplying \eqref{phiequ} by $\left(-\frac{\phi_{z}}{U_\epsilon^2} \right)_z$, integrating the resultant equation with respect to $z$, noting
\begin{equation}\nonumber
	 \int \phi_t\left(-\frac{\phi_{z}}{U_\epsilon^2} \right)_z=\int \frac{\phi_{z}\phi_{zt}}{U_\epsilon^2}=\frac{1}{2}\frac{\mathrm{d}}{\mathrm{d}t}\int\frac{\phi_{z}^2}{U_\epsilon^2},
\end{equation}
and	
\begin{equation}\nonumber
	\begin{aligned}
	 -\int \left(\frac{\phi_{z}}{U} \right)_z\left(-\frac{\phi_{z}}{U_\epsilon^2} \right)_z&=\int \left( \frac{\phi_{zz}}{U}-\frac{U_z\phi_{z}}{U^2}\right) \left( \frac{\phi_{zz}}{U_\epsilon^2}-\frac{2U_{\epsilon z}}{U_\epsilon^3}\phi_{z}\right)\\&=\int \frac{\phi_{zz}^2}{UU_\epsilon^2}-2\int \frac{U_{z}}{UU_\epsilon^3}\phi_{z}\phi_{zz}-\int \frac{U_z}{U^2U_\epsilon^2}\phi_{z}\phi_{zz}+2\int \frac{U_z^2}{U^2U_\epsilon^3}\phi_{z}^2,
	 \end{aligned}
\end{equation}	
we get
\begin{equation}\label{eq11}
	 \begin{aligned}
\frac{1}{2}\frac{\mathrm{d}}{\mathrm{d}t}\int\frac{\phi_{z}^2}{U_\epsilon^2}+\int \frac{\phi_{zz}^2}{UU_\epsilon^2}=&2\int \frac{U_{z}}{UU_\epsilon^3}\phi_{z}\phi_{zz}+\int \frac{U_z}{U^2U_\epsilon^2}\phi_{z}\phi_{zz}-2\int \frac{U_z^2}{U^2U_\epsilon^3}\phi_{z}^2\\&-\int g'(U)\phi_z	\left(-\frac{\phi_{z}}{U_\epsilon^2} \right)_z+\int F \left(-\frac{\phi_{z}}{U_\epsilon^2} \right)_z +\int G_z \left(-\frac{\phi_{z}}{U_\epsilon^2} \right)_z.
	 \end{aligned}
\end{equation}

Next, we estimate the terms on the right hand side of \eqref{eq11}, which exhibit the singularity due to $u_+=0$. To control the terms on the right hand side of \eqref{eq11}, special attention is required as $z\rightarrow+\infty$. We only need to consider two cases, i.e. $f'(0)<s<f'(u_-)$ and $f'(0)=s<f'(u_-)$, as there is no singularity when $z\rightarrow-\infty$. Therefore, noting $\frac{1}{U_\epsilon}\leq \frac{1}{U}$, by \eqref{|Uz|estimate}, \eqref{|Uz|estimate1} and the Cauchy-Schwarz inequality, we have
\begin{equation}\label{eq12}
	\begin{aligned}
2\int \frac{U_{z}}{UU_\epsilon^3}\phi_{z}\phi_{zz}+\int \frac{U_z}{U^2U_\epsilon^2}\phi_{z}\phi_{zz}&\leq \frac{1}{4}\int \frac{\phi_{zz}^2}{UU_\epsilon^2}+C\int \left( \frac{U_{z}^2}{UU_\epsilon^4}+\frac{U_z^2}{U^3U_\epsilon^2}\right)	\phi_{z}^2\\&\leq	\frac{1}{4}\int \frac{\phi_{zz}^2}{UU_\epsilon^2}+C \int\frac{\phi_{z}^2}{U}.
	\end{aligned}
\end{equation}
Thanks to $|g'(U)|\leq C$ and $|g''(U)|\leq C$ due to \eqref{g} and $f\in C^{\max\{k_{\pm}+1, 3\}}(\mathbb{R})$, we derive that
\begin{equation}\nonumber
	\begin{aligned}
-\int g'(U)\phi_z	\left(-\frac{\phi_{z}}{U_\epsilon^2} \right)_z&=-\int g''(U)\frac{U_z}{U_\epsilon^2}\phi_{z}^2-\int g'(U)\frac{\phi_{z}\phi_{zz}}{U_\epsilon^2}\\&\leq \frac{1}{4}\int \frac{\phi_{zz}^2}{UU_\epsilon^2}+C\int \left( 1+g'(U)^2\frac{U}{U_\epsilon^2}\right)  \phi_{z}^2\\&\leq \frac{1}{4}\int \frac{\phi_{zz}^2}{UU_\epsilon^2}+C\int \frac{\phi_{z}^2}{U}.
\end{aligned}
\end{equation}	
For the last two terms on the right hand side of \eqref{eq11}, in view of \eqref{|Uz|estimate} and \eqref{|Uz|estimate1}, it follows from \eqref{F1} and the Cauchy-Schwarz inequality that
\begin{equation}\nonumber
	\begin{aligned}
	\int F \left(-\frac{\phi_{z}}{U_\epsilon^2} \right)_z&=-\int F\left( \frac{\phi_{zz}}{U_\epsilon^2}-\frac{2U_{z}}{U_\epsilon^3}\phi_{z}\right)\\
	&\leq C \left(\int\frac{\phi_z^2|\phi_{zz}|}{U_\epsilon^2}+\int\frac{|U_{z}|}{U_\epsilon^3}|\phi_z^3| \right)\\
	& \leq N(t)\int \frac{\phi_{zz}^2}{UU_\epsilon^2}+C \int\left(U+1\right)\phi_{z}^2,
	\end{aligned}
\end{equation}
where we have used the fact that $\left\|\frac{\phi_z }{U}\right\|_{L^{\infty}}\leq CN(t)$ of \eqref{sobolev embedding} in the second inequality. By \eqref{Gz}, the last term on the right hand side of \eqref{eq11} can be estimated as
\begin{equation}\label{eq13}
	\begin{aligned}
\int G_z \left(-\frac{\phi_{z}}{U_\epsilon^2} \right)_z&=-\int G_z \left( \frac{\phi_{zz}}{U_\epsilon^2}-\frac{2U_{z}}{U_\epsilon^3}\phi_{z}\right)\\
&\leq C \left[\int \left(\frac{|U_z|}{U^3U_\epsilon^2}+ \frac{|U_{z}|}{U^2U_\epsilon^3}\right) \phi_{z}^2|\phi_{zz}|+\int \frac{U_z^2}{U^3U_\epsilon^3}|\phi_{z}^3|+\int\frac{|\phi_{z}|\phi_{zz}^2}{U^2U_\epsilon^2} \right]\\
&\leq CN(t)\left[\int\left(\frac{1}{U_\epsilon^2}+ \frac{1}{UU_\epsilon}\right)|\phi_{z}||\phi_{zz}|+\int \frac{\phi_z^2}{U}+\int\frac{\phi_{zz}^2}{UU_\epsilon^2} \right]\\
&\leq CN(t)\int \frac{\phi_{zz}^2}{UU_\epsilon^2}+C \int \frac{ \phi_{z}^2}{U}.
	\end{aligned}
\end{equation}

Substituting \eqref{eq12}--\eqref{eq13} into \eqref{eq11} and integrating the equation with respect to $t$, noting $U\leq \frac{C}{U}$, by \eqref{L2 estimate}, \eqref{L2 estimate1} and \eqref{L2 estimate2}, we have
\begin{equation}\nonumber
	\begin{aligned}
		\frac{1}{2}\int\frac{\phi_{z}^2}{U_\epsilon^2}+\left(\frac{1}{2}-CN(t)\right)\int_0^t\int \frac{\phi_{zz}^2}{UU_\epsilon^2}\leq\frac{1}{2}\int\frac{\phi_{0z}^2}{U_\epsilon^2}+C\left( \int\phi_0^2+\int w_4(U)\phi_0^2\right),
	\end{aligned}
\end{equation}
which, by virtue of $\frac{1}{U_\epsilon^2}\leq \frac{1}{U^2}$, further gives
\begin{equation}\nonumber
		\int\frac{\phi_{z}^2}{U_\epsilon^2}+\int \frac{\phi_{zz}^2}{UU_\epsilon^2}\leq C\left(\int\frac{\phi_{0z}^2}{U^2}+\int\phi_0^2+\int w_4(U)\phi_0^2 \right),
\end{equation}
provided that $N(t)$ is small enough, where $C$ is independent of $\epsilon$. Moreover, it follows from the Fatou's Lemma that
\begin{equation}\label{eq17}
	\begin{aligned}
		\quad\int\frac{\phi_{z}^2}{U^2}+\int \frac{\phi_{zz}^2}{U^3}&\leq \varliminf_{\epsilon\rightarrow0 }
\left(\int\frac{\phi_{z}^2}{U_\epsilon^2}+\int \frac{\phi_{zz}^2}{UU_\epsilon^2}\right)\\&\leq C\left(\int\frac{\phi_{0z}^2}{U^2}+\int\phi_0^2+\int w_4(U)\phi_0^2 \right).		
	\end{aligned}
\end{equation}
Thus \eqref{H1 estimate} is proved and we finish the proof of Lemma \ref{H1}.	
\end{proof}

To close the proof of \textit{a priori} estimates, we give the estimate of the second order derivative of $\phi$.
\begin{lemma}\label{H2}
	Let the assumptions of Proposition \ref{proposition priori estimate} hold. If $N(T)\ll1$, then it holds for any $t\in[0,T]$ that
	\begin{equation}\label{H2 estimate}
		\|\phi_{zz}\|_{w_1}^2+\int_0^t\left\|\phi_{zzz}(\tau)\right\|_{w_3}^2 \leq C\left(\|\phi_0\|^2+\|\phi_0\|_{w_4}^2+\|\phi_{0z}\|_{1,w_1}^2\right).
	\end{equation}	 	
\end{lemma}

\begin{proof}
We differentiate \eqref{phiequ} with respect to $z$ to get	
\begin{equation}\label{phizequ}
	\phi_{zt}+g''(U)U_z \phi_z+g'(U)\phi_{z z}-\left(\frac{\phi_z}{U}\right)_{zz}=F_z+G_{zz}.
\end{equation}	
Multiplying \eqref{phizequ} by $\left(-\frac{\phi_{zz}}{U_\epsilon^2} \right)_z$, integrating the resultant equations with respect to $z$ and using
\begin{equation}\nonumber
	\int \phi_{zt}\left(-\frac{\phi_{zz}}{U_\epsilon^2} \right)_z=\int \frac{\phi_{zz}\phi_{zzt}}{U_\epsilon^2}=\frac{1}{2}\frac{\mathrm{d}}{\mathrm{d}t}\int\frac{\phi_{zz}^2}{U_\epsilon^2},
\end{equation}
and	
\begin{equation}\nonumber
	\begin{aligned}
		-\int \left(\frac{\phi_{z}}{U} \right)_{zz}\left(-\frac{\phi_{zz}}{U_\epsilon^2} \right)_z&=\int \left( \frac{\phi_{zzz}}{U}-\frac{2U_z}{U^2}\phi_{zz}+\frac{2U_z^2}{U^3}\phi_{z}-\frac{U_{zz}}{U^2}\phi_{z}\right)\left( \frac{\phi_{zzz}}{U_\epsilon^2}-\frac{2U_{z}}{U_\epsilon^3}\phi_{zz}\right)
		\\&=\int \frac{\phi_{zzz}^2}{UU_\epsilon^2}-2\int\left(\frac{U_{\epsilon z}}{UU_\epsilon^3}+\frac{U_z}{U^2U_\epsilon^2} \right) \phi_{zz}\phi_{zzz}+4\int \frac{U_z^2}{U^2U_\epsilon^3}\phi_{zz}^2\\&\quad+\int\left( \frac{2U_z^2}{U^3U_\epsilon^2}-\frac{U_{zz}}{U^2U_\epsilon^2}\right) \phi_{z}\phi_{zzz}-2\int\left( \frac{2U_z^3}{U^3U_\epsilon^3}-\frac{U_{z}U_{zz}}{U^2U_\epsilon^3}\right) \phi_{z}\phi_{zz},
	\end{aligned}
\end{equation}	
we get
\begin{equation}\label{eq14}
	\begin{aligned}
	&\frac{1}{2}\frac{\mathrm{d}}{\mathrm{d}t}\int\frac{\phi_{zz}^2}{U_\epsilon^2}+\int \frac{\phi_{zzz}^2}{UU_\epsilon^2}+4\int \frac{U_z^2}{U^2U_\epsilon^3}\phi_{zz}^2\\&=\int (g''(U)U_z\phi_z+g'(U)\phi_{zz})\left(-\frac{\phi_{zz}}{U_\epsilon^2} \right)_z+2\int\left(\frac{U_{z}}{UU_\epsilon^3}+\frac{U_z}{U^2U_\epsilon^2} \right) \phi_{zz}\phi_{zzz}\\&\quad-\int\left( \frac{2U_z^2}{U^3U_\epsilon^2}-\frac{U_{zz}}{U^2U_\epsilon^2}\right) \phi_{z}\phi_{zzz}+2\int\left( \frac{2U_z^3}{U^3U_\epsilon^3}-\frac{U_{z}U_{zz}}{U^2U_\epsilon^3}\right) \phi_{z}\phi_{zz}\\&\quad-\int F_z\left(\frac{\phi_{zz}}{U_\epsilon^2} \right)_z-\int G_{zz}\left(\frac{\phi_{zz}}{U_\epsilon^2} \right)_z.
	\end{aligned}
\end{equation}
In view of \eqref{ODE}, one obtains
\begin{equation}\nonumber
	\begin{aligned}
	 U_{zz}&=U_z\left[f(U)-f(u_{ \pm})-s(U-u_{ \pm})\right]+U(f'(U)-s)U_z\\&=\frac{U_z^2}{U}+(f'(U)-s)UU_z,
	 \end{aligned}
\end{equation}
which, in combination with \eqref{|Uz|estimate}, implies when $f'(0)<s<f'(u_-)$ that
\begin{equation}\label{Uzz}
	 |U_{zz}(z)|\leq C\left(\frac{ |U_z(z)|}{U(z)}+U(z) \right)|U_z(z)| \leq CU^3(z)\quad \text{for all }z\in (-\infty,+\infty),
\end{equation}
and with \eqref{|Uz|estimate1} implies for $f'(0)=s$ that
\begin{equation}\label{Uzz1}
	|U_{zz}(z)|\leq C\left(\frac{ |U_z(z)|}{U(z)}+U(z) \right)|U_z(z)| \leq CU^3(z),\quad \text{as }z\rightarrow+\infty.
\end{equation}
Now by \eqref{g}, \eqref{Uzz}, \eqref{Uzz1} and $f\in C^{\max\{k_{\pm}+1, 3\}}(\mathbb{R})$, we get by integration by parts that
\begin{equation}\label{eq15}
	\begin{aligned}
		&\quad\int (g''(U)U_z\phi_z+g'(U)\phi_{zz})\left(-\frac{\phi_{zz}}{U_\epsilon^2} \right)_z\\&=-\int\left(g'''(U)\frac{U_z^2}{U_\epsilon^2}+g''(U)\frac{U_{zz}}{U_\epsilon^2}\right)\phi_{z}\phi_{zz}-2\int	g''(U)\frac{U_{z}}{U_\epsilon^2} \phi_{zz}^2-\int g'(U)\frac{\phi_{zz}\phi_{zzz}}{U_\epsilon^2}\\
		&\leq \frac{1}{8}\int \frac{\phi_{zzz}^2}{UU_\epsilon^2}+C\left[\int \left(g'(U)^2\frac{U}{U_\epsilon^2}+1\right)\phi_{zz}^2+\int \phi_{z}^2\right]\\
		&\leq\frac{1}{8}\int \frac{\phi_{zzz}^2}{UU_\epsilon^2}+C\left(\int\frac{\phi_{z z}^2}{U}+\int\frac{\phi_{z }^2}{U} \right).
	\end{aligned}
\end{equation}
Similarly, thanks to \eqref{|Uz|estimate}, \eqref{|Uz|estimate1} and the Cauchy-Schwarz inequality, we deduce that
\begin{equation}\nonumber
	 \begin{aligned}
&\quad2\int\left(\frac{U_{\epsilon z}}{UU_\epsilon^3}+\frac{U_z}{U^2U_\epsilon^2} \right) \phi_{zz}\phi_{zzz}-\int\left( \frac{2U_z^2}{U^3U_\epsilon^2}-\frac{U_{zz}}{U^2U_\epsilon^2}\right) \phi_{z}\phi_{zzz}\\&\leq C\left( \int\frac{|\phi_{zz}||\phi_{zzz}| }{U_\epsilon^2}+\int\frac{|\phi_{z}||\phi_{zzz}| }{U_\epsilon}\right)\\&\leq \frac{1}{4}	\int \frac{\phi_{zzz}^2}{UU_\epsilon^2}+C\left(\int \frac{U}{U_\epsilon^2}\phi_{z z}^2+\int U\phi_{z}^2\right)\\&\leq	 \frac{1}{4}	\int \frac{\phi_{zzz}^2}{UU_\epsilon^2}+C\left(\int \frac{\phi_{z z}^2}{U}+\int \frac{\phi_{z}^2}{U}\right),
	 \end{aligned}
\end{equation}
and
\begin{equation}\nonumber
	\begin{aligned}
		2\int\left( \frac{2U_z^3}{U^3U_\epsilon^3}-\frac{U_{z}U_{zz}}{U^2U_\epsilon^3}\right) \phi_{z}\phi_{zz}\leq C\left( \int\frac{\phi_{z z}^2}{U}+\int\frac{\phi_{z }^2}{U}\right).
	\end{aligned}
\end{equation}
For the last second term on the right hand side of \eqref{eq14}, noting $\left\|\frac{\phi_z }{U}\right\|_{L^{\infty}}\leq CN(t)$, it follows from \eqref{Fz} that
\begin{equation}\nonumber
	\begin{aligned}
\int F_z\left(-\frac{\phi_{zz}}{U_\epsilon^2} \right)_z&=-\int F_z	\left( \frac{\phi_{zzz}}{U_\epsilon^2}-\frac{2U_{z}}{U_\epsilon^3}\phi_{zz}\right)\\	
&\leq C \int \left(\frac{|U_z|}{U_\epsilon^2}\phi_{z}^2|\phi_{zzz}|+\frac{|U_z||U_{z}|}{U_\epsilon^3}\phi_{z}^2|\phi_{zz}|
+\frac{|\phi_{z}||\phi_{zz}||\phi_{zzz}|}{U_\epsilon^2}+\frac{|U_{z}|}{U_\epsilon^3}|\phi_{z}|\phi_{zz}^2\right)\\
&\leq CN(t) \int\left( |\phi_{z}||\phi_{zzz}|+|\phi_{z}||\phi_{zz}|+\frac{|\phi_{zz}||\phi_{zzz}| }{U_\epsilon}+\phi_{zz }^2\right)\\
& \leq N(t) \int \frac{\phi_{zzz}^2}{UU_\epsilon^2}+C\left[\int\left(U+1\right)\phi_{z z}^2  +\int\left( UU_\epsilon^2+1\right)\phi_{z }^2  \right].
	\end{aligned}
\end{equation}
Thanks to \eqref{Gzz}, the last term on the right hand side of \eqref{eq14} can be estimated as
\begin{eqnarray}\label{eq27}
\int G_{zz}\left(-\frac{\phi_{zz}}{U_\epsilon^2} \right)_z
&=&	-\int G_{zz} \left( \frac{\phi_{zzz}}{U_\epsilon^2}-\frac{2U_{z}}{U_\epsilon^3}\phi_{zz}\right) \nonumber \\
&\leq& C	 \bigg[\int\frac{\phi_{zz}^2|\phi_{zzz}|}{U^2U_\epsilon^2}+\int\frac{|U_{z}|}{U^2U_\epsilon^3} |\phi_{zz}^3|+\int\left(\frac{U_z^2}{U^4U_\epsilon^2}+\frac{|U_{zz}|}{U^3U_\epsilon^2} \right)\phi_{z}^2|\phi_{zzz}|\nonumber\\
&&+\int\left(\frac{|U_{z}|^3}{U^4U_\epsilon^3}+\frac{|U_{z}||U_{zz}|}{U^3U_\epsilon^3} \right)\phi_{z}^2|\phi_{zz}|+\int\left(\frac{|U_z|}{U^3U_\epsilon^2}+ \frac{|U_{z}|}{U^2U_\epsilon^3}\right) |\phi_{z}||\phi_{zz}||\phi_{z z z}| \nonumber\\
&&+\int \frac{|U_z|^2}{U^3U_\epsilon^3}|\phi_{z}|\phi_{zz}^2+\int \frac{|\phi_{z}|\phi_{zzz}^2}{U^2U_\epsilon^2}\bigg] \nonumber\\
&=:& I_1+I_2+I_3+I_4+I_5+I_6+I_7.
\end{eqnarray}
By virtue of $\|U^{-1}\phi_{zz}\|_{L^2}\leq N(t)$ and H\"{o}lder's inequality, it holds that
\begin{equation}\nonumber
\begin{aligned}
	I_1=\int\frac{\phi_{zz}^2|\phi_{zzz}|}{U^2U_\epsilon^2}
	 &\leq \|U^{-\frac{1}{2}}U_{\epsilon}^{-1}\phi_{zz}\|_{L^{\infty}}\|U^{-1}\phi_{zz}\|_{L^2}\|U^{-\frac{1}{2}}U_{\epsilon}^{-1}\phi_{zzz}\|_{L^2} \\
	 &\leq CN(t)\|U^{-\frac{1}{2}}U_{\epsilon}^{-1}\phi_{zz}\|_{L^{2}}^{\frac{1}{2}}\|U^{-\frac{1}{2}}U_{\epsilon}^{-1}\phi_{zzz}\|_{L^{2}}^{\frac{3}{2}}\\
	 &\leq N(t)\int \frac{\phi_{zzz}^2}{UU_\epsilon^2}+C\int  \frac{\phi_{z z}^2}{U^3},
\end{aligned}	
\end{equation}
and that,
\begin{equation}\nonumber
\begin{aligned}
I_2=\int\frac{|U_{z}|}{U^2U_\epsilon^3} |\phi_{zz}^3|\leq C\int\frac{|\phi_{zz}^3|}{U_\epsilon^3}
    &\leq\|U^{-\frac{1}{2}}U_{\epsilon}^{-1}\phi_{zz}\|_{L^{\infty}}\|U^{\frac{1}{2}}U_{\epsilon}^{-\frac{1}{2}}\phi_{zz}\|_{L^{2}}\|U^{-\frac{3}{2}}\phi_{zz}\|_{L^2}\\
	&\leq CN(t)\|U^{-\frac{1}{2}}U_{\epsilon}^{-1}\phi_{zzz}\|_{L^{2}}^{\frac{1}{2}}\|U^{-\frac{1}{2}}U_{\epsilon}^{-1}\phi_{zz}\|_{L^{2}}^{\frac{1}{2}}\|U^{-\frac{3}{2}}\phi_{zz}\|_{L^2}\\
	&\leq N(t)\int \frac{\phi_{zzz}^2}{UU_\epsilon^2}+C\int  \frac{\phi_{z z}^2}{U^3},
\end{aligned}
\end{equation}
where we have used $\|\phi_{zz}\|_{L^2}\leq C\|U^{-1}\phi_{zz}\|_{L^2}\leq N(t)$ in the third inequality. Furthermore, we utilize \eqref{|Uz|estimate}, \eqref{|Uz|estimate1} and the Cauchy-Schwarz inequality to get
\begin{eqnarray}\label{eq16}
		&&\quad I_3+I_4+I_5+I_6+I_7 \notag \\
        && \leq CN(t)\left(\int \frac{|\phi_{z}||\phi_{zzz}|}{U_\epsilon}+\int|\phi_{z}||\phi_{zz}|+\int\frac{ |\phi_{zz}||\phi_{zzz}|}{U_\epsilon^2}+\int \frac{\phi_{zz}^2}{U}+\int \frac{\phi_{zzz}^2}{UU_\epsilon^2}\right)\notag \\
		&&\leq CN(t)\int \frac{\phi_{zzz}^2}{UU_\epsilon^2}+C\left(\int  \frac{\phi_{z z}^2}{U}+\int U\phi_{z }^2 \right),
\end{eqnarray}
where we have used the fact $\left\|\frac{\phi_z }{U}\right\|_{L^{\infty}}\leq CN(t)$ in the first inequality. Noting that $U\leq \frac{C}{U}\leq \frac{C}{U^3}$ and $1\leq \frac{C}{U}$, integrating the equation \eqref{eq14} with respect to $t$ and combining \eqref{eq15}--\eqref{eq16}, \eqref{L2 estimate}, \eqref{L2 estimate1} and \eqref{H1 estimate}, we get
\begin{equation}\nonumber
	\begin{aligned}
		\int\frac{\phi_{zz}^2}{UU_\epsilon}+\int_{0}^{t}\int \frac{ \phi_{zzz}^2 }{UU_\epsilon^2} &\leq\int\frac{\phi_{0zz}^2}{UU_\epsilon}+C\left(\int_{0}^{t}\int\frac{\phi_{z z}^2}{U^3}+\int_{0}^{t}\int  \frac{\phi_{z}^2}{U}\right)\\&\leq C\left(\int\frac{\phi_{0zz}^2}{U^2}+\int  \frac{\phi_{0z}^2}{U^2}+\int\phi_0^2+\int w_4(U)\phi_0^2\right),
	\end{aligned}	
\end{equation}
provided $N(t)$ is suitably small.  Following Fatou's Lemma as in the proof of \eqref{eq17}, we then obtain \eqref{H2 estimate} and finish the proof of Lemma \ref{H2}.
\end{proof}

\begin{proof}[Proof of Proposition \ref{proposition priori estimate}]
The desired estimate \eqref{priori estimate} follows from \eqref{L2 estimate}, \eqref{L2 estimate1}, \eqref{H1 estimate} and \eqref{H2 estimate}.\end{proof}

\begin{proof}[Proof of Theorem \ref{phi stability}]
Recalling from the proof of Lemma \ref{L2}, we have selected $\delta_{2}$ such that $\delta_{2}\leq\delta_0$, where $\delta_0$ is taken arbitrarily and mentioned in Proposition \ref{local existence}. Let $\delta_{1}=\min\left\{\delta_2/2, \delta_2/2C\right\}$, and $N(0)\leq\delta_{1}$. Since $N(0)\leq\delta_{1}\leq \delta_2/2<\delta_0$, then Proposition \ref{local existence} guarantees the existence of a unique solution $\phi\in X(0,T_0)$ for the system \eqref{phiequ} satisfying $N(t)\leq 2N(0)\leq 2\delta_{1}\leq2\delta_2/2=\delta_2$ for any $0\leq t\leq T_0$. Subsequently, by employing Proposition \ref{proposition priori estimate}, we establish that $N(T_0)\leq CN(0)\leq C\delta_{1}\leq C\delta_2/{2C}=\delta_2/2$. Now  considering the system \eqref{phiequ} with the \lq\lq initial data\rq\rq at $T_0$, and utilizing Proposition \ref{local existence} once again, we obtain a unique solution $\phi\in X(T_0,2T_0)$, eventually on the interval $[0,2T_0]$, satisfying $N(t)\leq 2N(T_0)\leq 2\delta_2/2=\delta_2$ for any $0\leq t\leq 2T_0$. Then, applying Proposition \ref{proposition priori estimate} with $T=2T_0$ again, we deduce that $N(2T_0)\leq CN(0)\leq C\delta_{1}\leq C\delta_2/{2C}=\delta_2/2<\delta_0$. Hence, by repeating this continuation process, we can derive a unique global solution $\phi\in X(0,\infty)$ that satisfies the estimate \eqref{priori estimate} for all $t\in[0,\infty)$.

To complete the proof, it remains to show \eqref{asymptotic}. We denote $h(t)\triangleq\|\phi_z(\cdot,t)\|_1^2$. Given the established estimate \eqref{priori estimate} for all $t\in[0,\infty)$, it is easy to see that
$$
\int_{0}^{\infty}\|\phi_z(\cdot,\tau)\|_1^2 \mathrm{d}\tau \leq \int_{0}^{\infty}\left(\left\|\phi_z(\tau)\right\|_{w_2}^2+\left\|\phi_z(\tau)\right\|_{w_5}^2+\left\|\phi_{zz}(\tau)\right\|_{w_3}^2\right)\mathrm{d}\tau\leq C N^2(0) \leq C,
$$	
which implies $h(t)\in L^1(0,\infty)$. Furthermore, by combining \eqref{priori estimate} with \eqref{phizequ}, a direct  calculation yields
\begin{equation}\nonumber
	\begin{aligned}
	\int_{0}^{\infty } |h'(\tau)|\mathrm{d}\tau
		&=
		\int_{0}^{\infty } \left|\frac{\mathrm{d}}{\mathrm{d}\tau}\|\phi_z(\cdot,\tau)\|_1^2\right|\mathrm{d}\tau\\
&=\int_{0}^{\infty } \left|\frac{\mathrm{d}}{\mathrm{d}\tau}\int\left(\phi_{z}^2+\phi_{z z}^2 \right)	\right|\mathrm{d}\tau\\&=2 \int_{0}^{\infty } \left|\int\left(\phi_z\phi_{zt}+\phi_{zz}\phi_{zzt}\right)\right|\mathrm{d}\tau\\& \leq C\int_{0}^{\infty }\int \left( \frac{\phi_{z}^2 }{U}+\frac{ \phi_{z z}^2}{U^3}+\frac{ \phi_{z zz}^2}{U^3}\right)\mathrm{d}\tau\\&\leq C,
	\end{aligned}
\end{equation}
where $\phi_{zzt}=-g^{\prime \prime \prime}(U) U_z^2 \phi_z-g^{\prime \prime}(U) U_{zz} \phi_z-2 g^{\prime \prime}(U) U_z \phi_{z z}-g^{\prime}(U) \phi_{z z z}+F_{z z}+\left(\left(\frac{\phi_z}{U}\right)_{z z }+G_{z z }\right)_z$ and
\[\int\phi_{zz}\left(\left(\frac{\phi_z}{U}\right)_{z z }+G_{z z }\right)_z=-\int\phi_{zzz}\left(\left(\frac{\phi_z}{U}\right)_{z z }+G_{z z }\right),\]
have been used. It then follows from the convergence lemma shown in the text book \cite{Matsumura-Nishihara-Textbook} that
\begin{equation}\nonumber
	\|\phi_z(\cdot,t)\|_1^2\rightarrow 0, \quad \text { as } t \rightarrow +\infty.
\end{equation}
This along with the H\"{o}lder's inequality implies that
\begin{equation}\nonumber
	\begin{aligned}
		\phi_z^2(z, t) & =2 \int_{-\infty}^z \phi_z \phi_{z z}(y, t) d y \\
		& \leqslant 2\left(\int_{-\infty}^{\infty} \phi_z^2 d y\right)^{1 / 2}\left(\int_{-\infty}^{\infty} \phi_{z z}^2 d y\right)^{1 / 2} \\
		& \leqslant C\left\|\phi_z(\cdot, t)\right\|_{1} \rightarrow 0,\quad \text { as } t \rightarrow+\infty.
	\end{aligned}
\end{equation}
Therefore \eqref{asymptotic} is proved and the proof of Theorem \ref{phi stability} is complete.
\end{proof}

\section{Numerical simulations}

In this section we carry out some numerical simulations. We numerically test three cases spanning non-degenerate and the degenerate shock profiles.

Case 1: $f'(0)<s<f'(u_-)$.
Let $f(u)=u^2/2$, $u_+=0$, $u_-=2$, then $s=\frac{f(u_-)-f(0)}{u_--0}=1$,  $f'(0)=0$ and $f'(u_-)=u_-=2$, such that $f'(0)<s<f'(u_-)$ holds (see Figure \ref{fig:case1}).
\begin{figure}[htbp]
	\begin{center}
		\includegraphics[width=7cm]{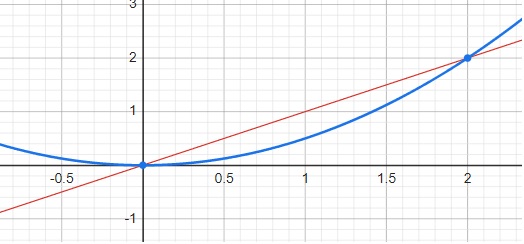}
		\caption{Case 1--the graph of $f(u)$}
		\label{fig:case1}
	\end{center}
\end{figure}
The initial data is chosen  as follows:
\begin{equation}\label{IV-2}
u_0(x)=
\begin{cases}
	2-\cos 8x \ e^{2x}, & \mbox{ for } x<0,\\
	\frac{1}{1+2x}, & \mbox{ for } x\ge 0.
\end{cases}
\end{equation}
\begin{figure}[htbp]
	\begin{center}
		\includegraphics[width=7cm]{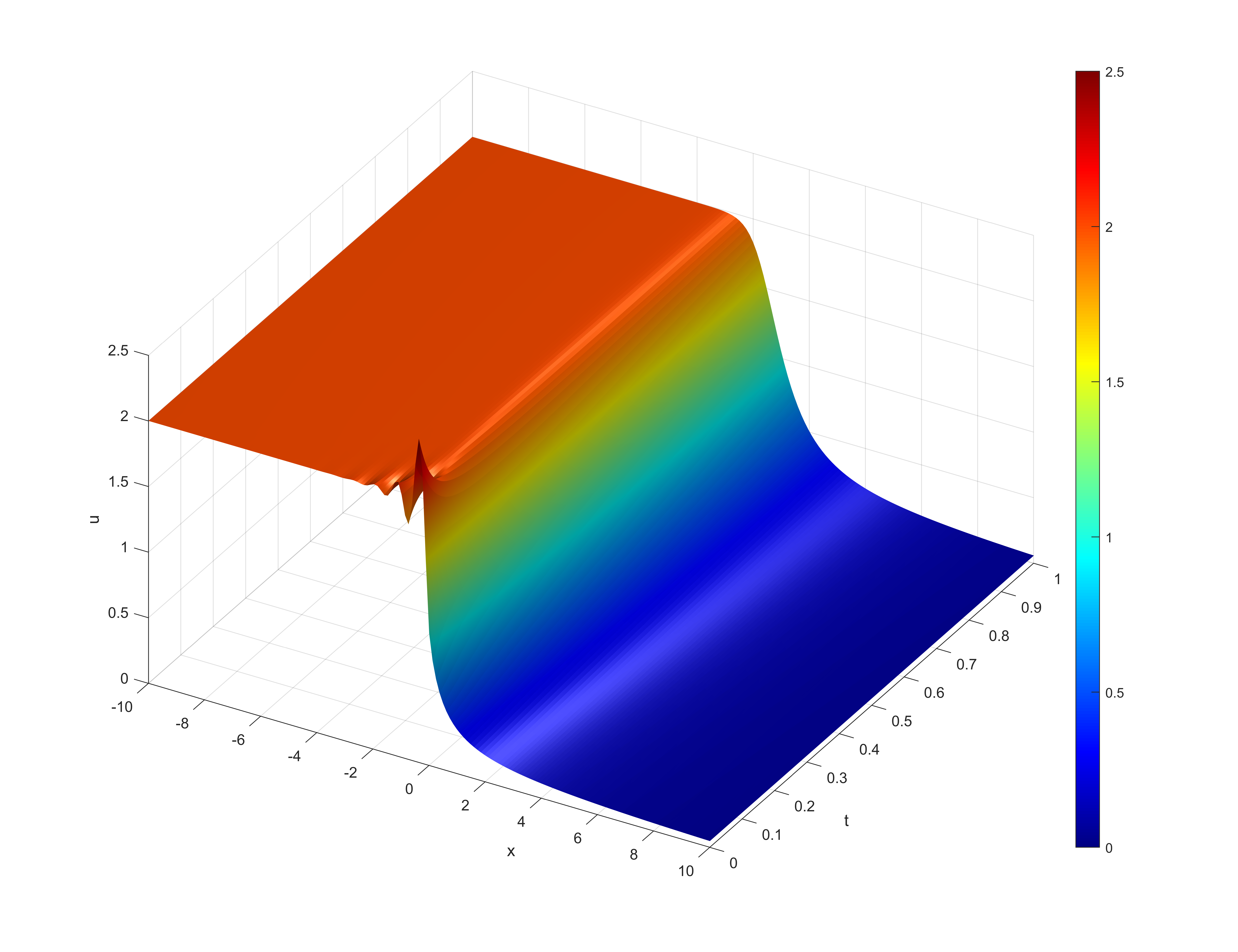}
		\includegraphics[width=7cm]{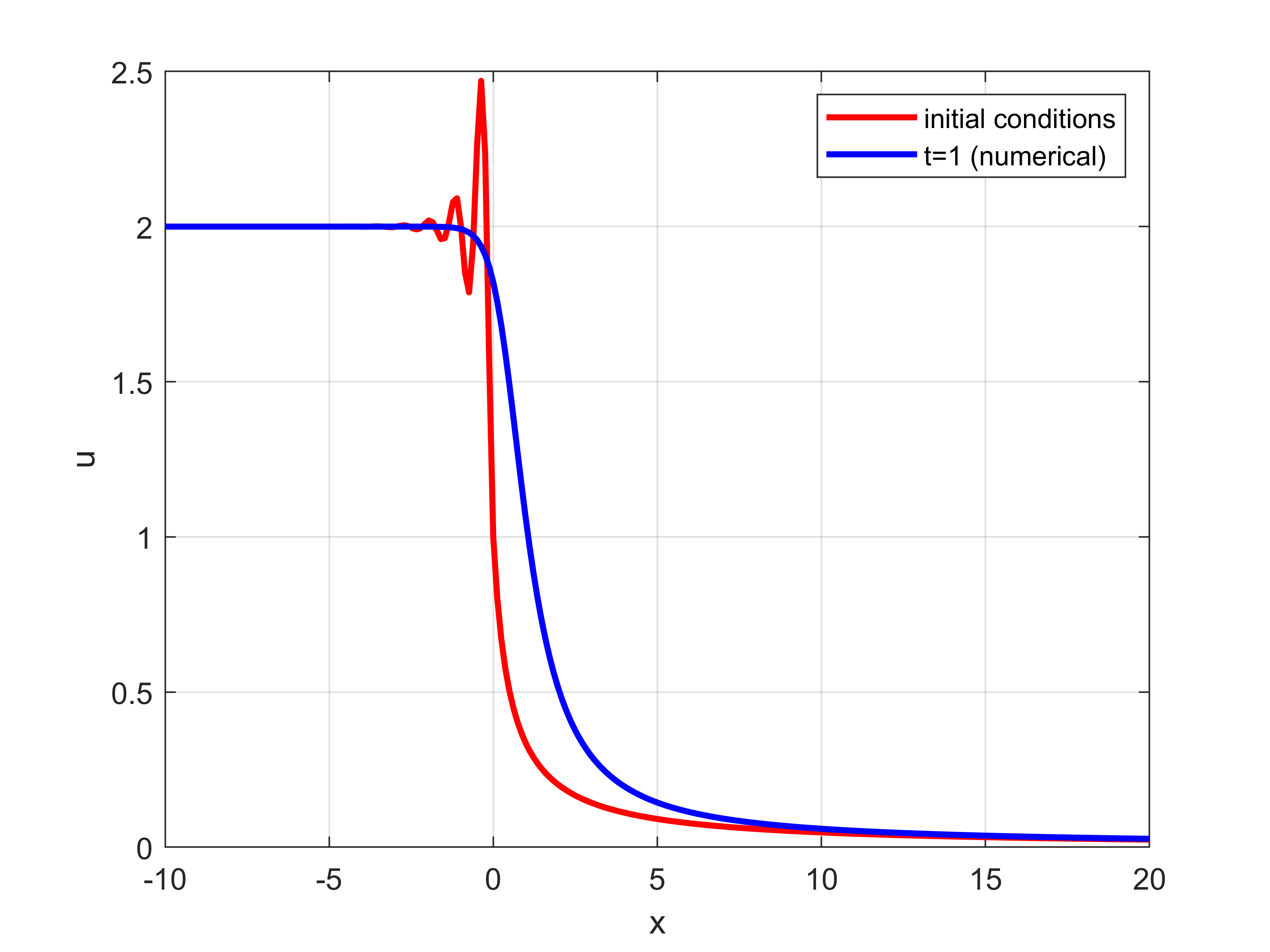}
		\caption{Case 1--the solution $u(x,t)$ behaves like a monotone viscous shock wave despite being initially perturbed}
		\label{fig1-2}
	\end{center}
\end{figure}
Figure \ref{fig1-2} demonstrates that initially, there are some oscillations for the solution $u(x,t)$. However, these oscillations quickly disappear due to the effect of strong diffusion, and finally the solution behaves like a viscous shock wave.

Case 2: $f'(0)=s<f'(u_-)$.
Let $f(u)=(u-1)^3+(u-1)^2+(u-1)$, $u_+=0$, $u_-=2$, then $s=\frac{f(u_-)-f(0)}{u_--0}=2$,  $f'(0)=2$ and $f'(u_-)=6$, such that $f'(0)=s<f'(u_-)$ holds (see Figure \ref{fig:case2}).
\begin{figure}[htbp]
	\begin{center}
		\includegraphics[width=7cm]{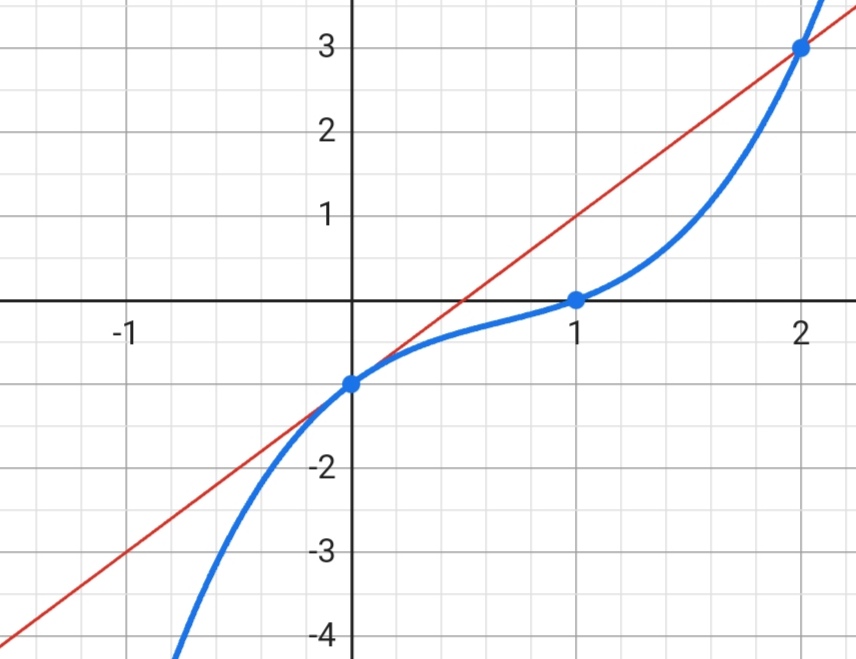}
		\caption{Case 2--graph of $f(u)$}
		\label{fig:case2}
	\end{center}
\end{figure}
Furthermore, let the initial data be
\[
u_0(x)=
\begin{cases}
	2, & \mbox{ for } x<0,\\
	\frac{2}{(1+x)^2}, & \mbox{ for } x\ge 0.
\end{cases}
\]
The numerical simulation (see Figure \ref{fig2-1}) indicates that the solution $u(x,t)$ behaves as expected like a viscous shock wave.
\begin{figure}
	\begin{center}
		\includegraphics[width=7cm]{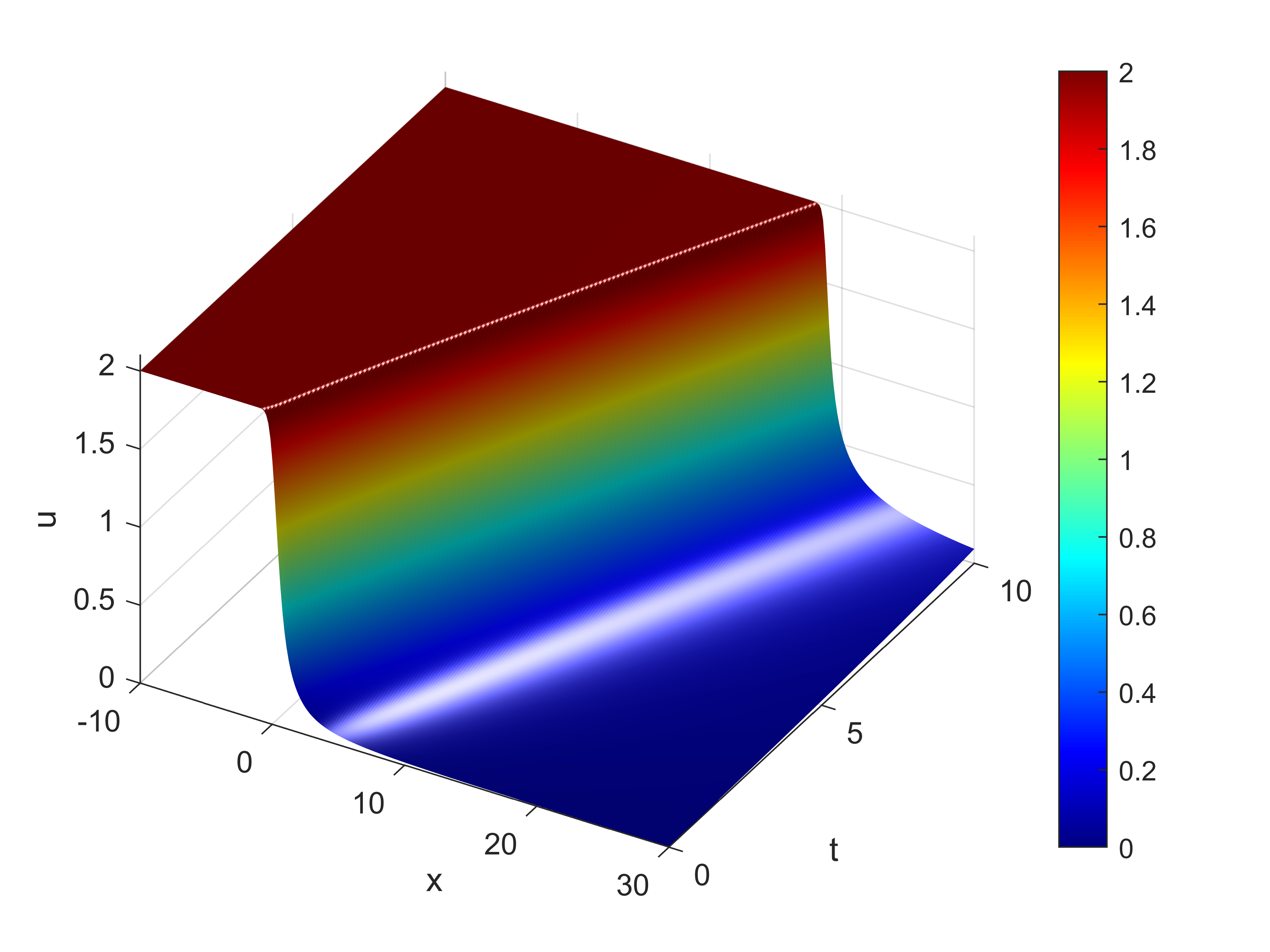}
        \includegraphics[width=7cm]{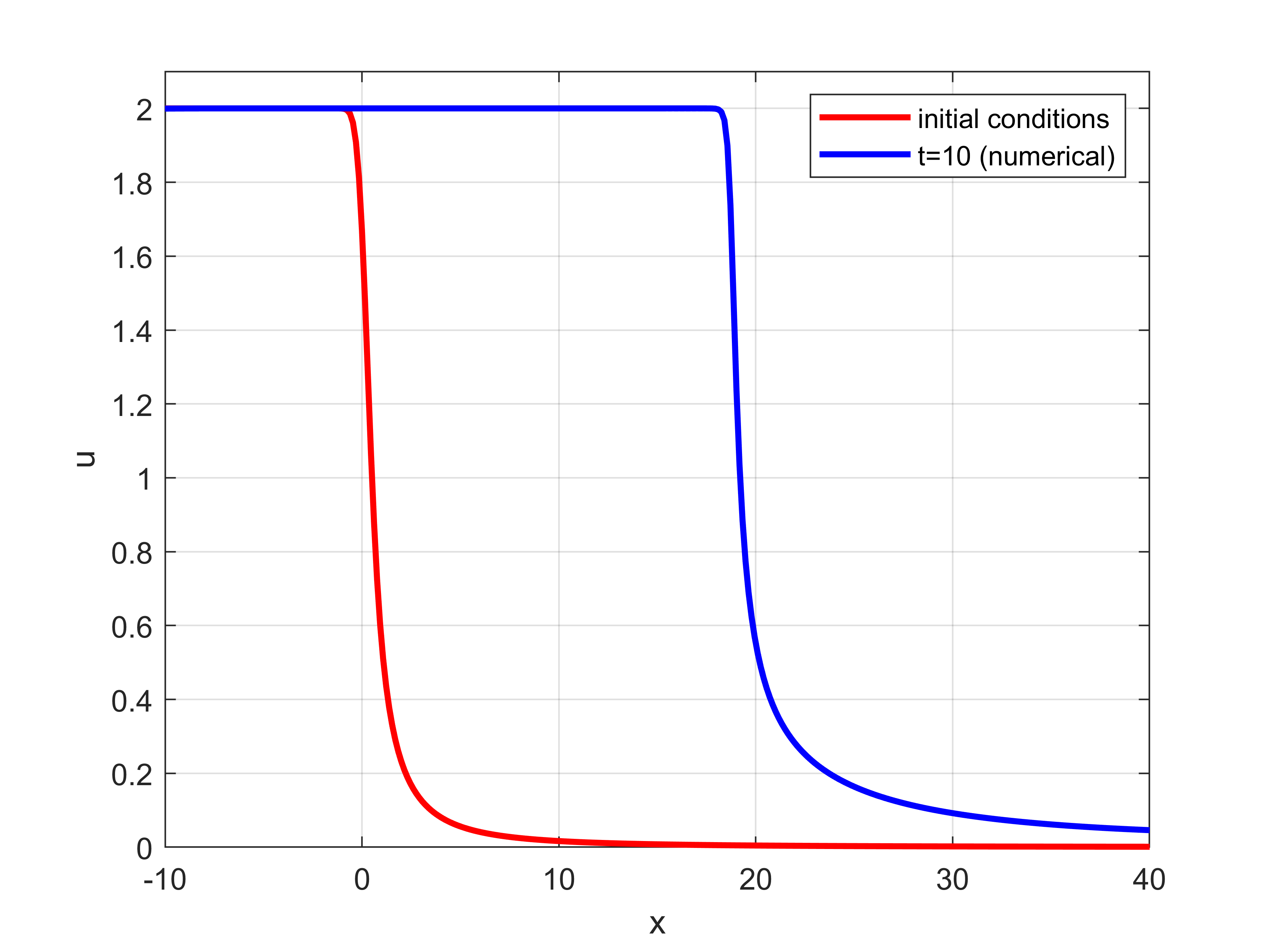}
		\caption{Case 2--the solution  $u(x,t)$ behaves like a monotone viscous shock wave}
		\label{fig2-1}
	\end{center}
\end{figure}

Case 3: $f'(0)<s=f'(u_-)$.
Let $f(u)=-u^3+2u^2-\frac{1}{2}u$, $u_+=0$, $u_-=1$, then $s=\frac{f(u_-)-f(0)}{u_--0}=\frac{1}{2}$,  $f'(0)=-\frac{1}{2}$ and $f'(u_-)=\frac{1}{2}$, such that $f'(0)<s=f'(u_-)$ holds (see Figure \ref{fig:case3}).
\begin{figure}[htbp]
	\begin{center}
		\includegraphics[width=7cm]{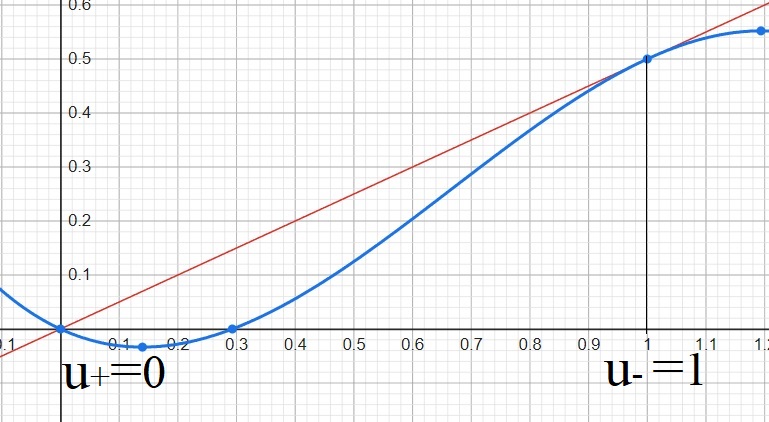}
		\caption{Case 3--the graph of $f(u)$}
		\label{fig:case3}
	\end{center}
\end{figure}
We take the initial data as
\[
u_0(x)=
\begin{cases}
	1-\frac{1}{1+|x|}, & \mbox{ for } x<0,\\
	0, & \mbox{ for } x\ge 0.
\end{cases}
\]
Figure \ref{fig3-1} shows that the solution $u(x,t)$ behaves indeed like a viscous shock wave.
\begin{figure}[htbp]
	\begin{center}
		\includegraphics[width=7cm]{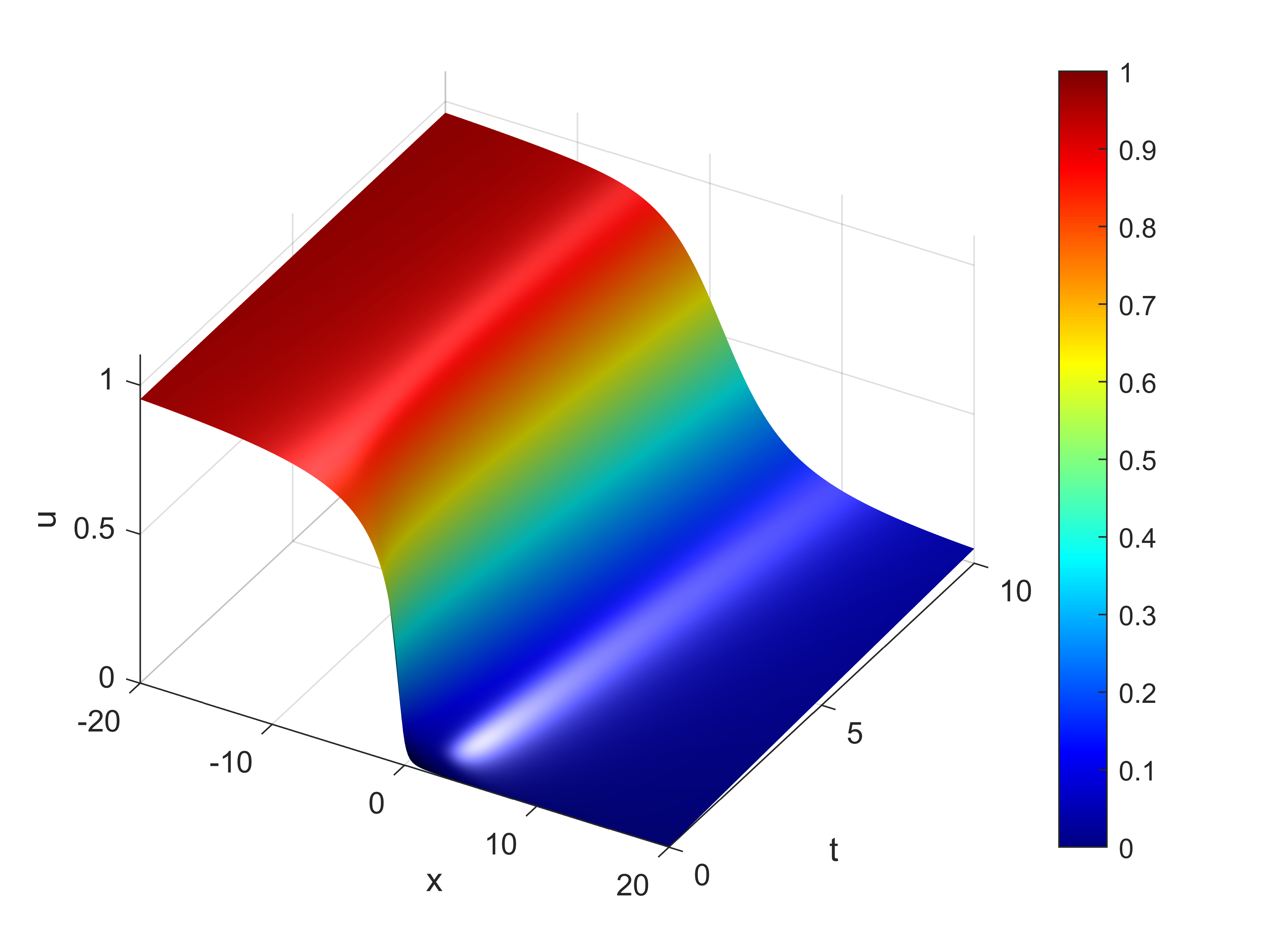}
        \includegraphics[width=7cm]{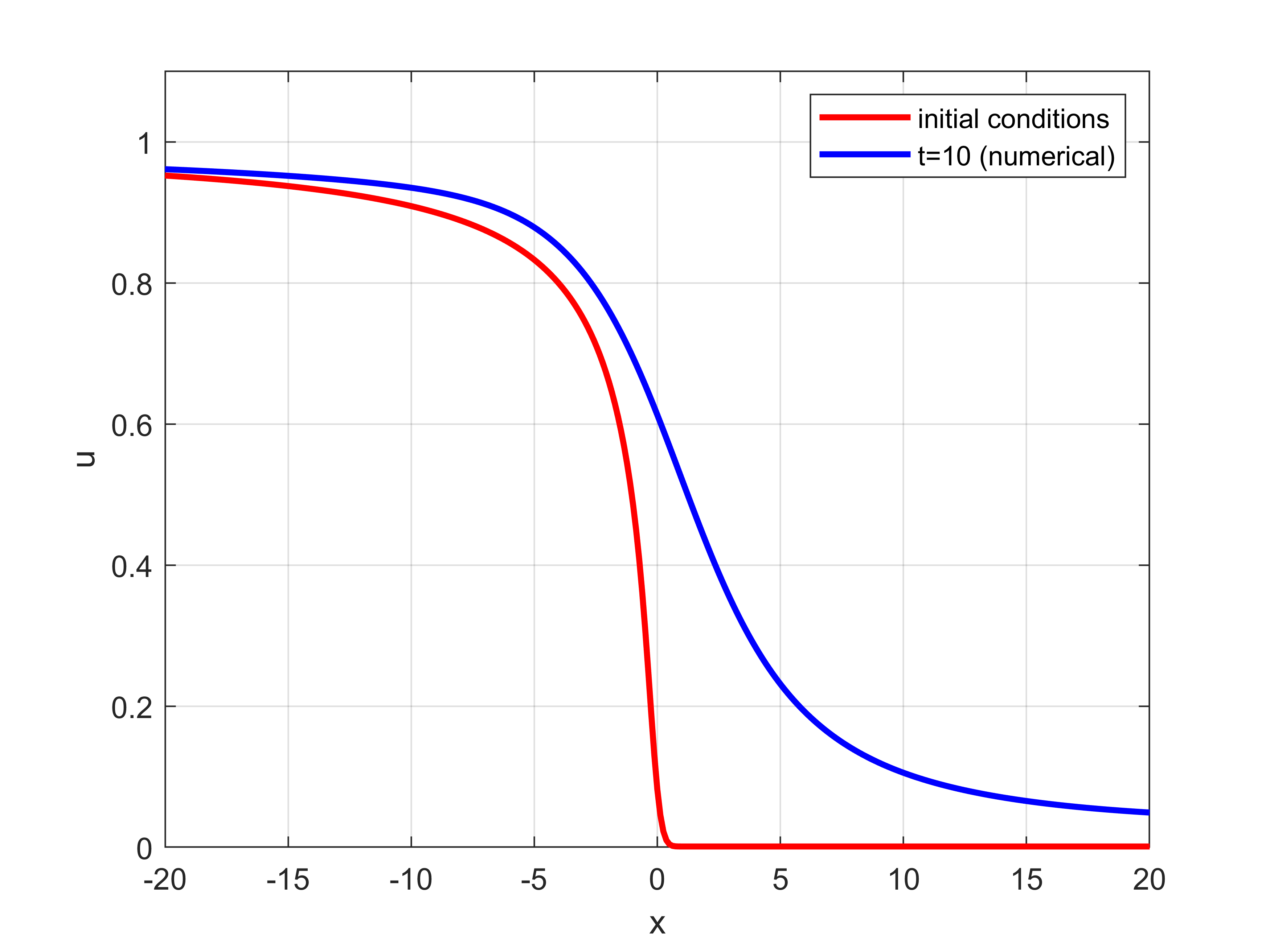}
		\caption{Case 3--the solution  $u(x,t)$ behaves like a monotone viscous shock wave}
		\label{fig3-1}
	\end{center}
\end{figure}

\section*{Acknowledgements} The work was done when X. Li visited McGill University as a PhD trainee supported by
China Scholarship Council (CSC) for the PhD trainee program (202206620041). She would like to express her
sincere thanks for the hospitality of McGill University and CSC. The research of J. Li was supported by the Natural Science Foundation of Jilin Province (No. 20210101144JC) and the National Natural Science Foundation of China (No. 12371216).  The research of M. Mei was supported by NSERC Grant RGPIN-2022-03374. The research of J.-C. Nave was supported by the NSERC Discovery Grant program.


\begin{thebibliography}{99}
	\bibitem{burgers1} J. M. Burgers, A mathematical model illustrating the theory of turbulence, Adv. Appl. Mech., 1 (1948), 171-199.
	
	\bibitem{burgers2} J. M. Burgers, Mathematical examples illustrating relations occurring in the theory of turbulent fluid motion, Verh. Kon. Ned. Akad. Wetensch., 2 (1939), 1-53.
	
	\bibitem{Beck} M. Beck and C. E. Wayne, Using global invariant manifolds to understand metastability in the Burgers equation with small viscosity, SIAM Rev., 53 (2011), 129-153.
	
	\bibitem{Engler} H. Engler,
	Asymptotic stability of traveling wave solutions for perturbations with algebraic decay, J. Differ. Equ., 185 (2002), 348-369.
	
	\bibitem{Serre} H. Freist\"uhler and D. Serre,
	$L^1$ stability of shock waves in scalar viscous conservation laws, Commun. Pure Appl. Math., 51 (1998), 291-301.
	
	\bibitem{Fries} C. Fries,
	Stability of viscous shock waves associated with non-convex modes, Arch. Ration. Mech. Anal., 152 (2000), 141-186.
	
	\bibitem{Galaktionv2} V. A. Galaktionv, Blow-up for quasilinear heat equations with critical Fujita's exponents, P. Roy. Soc. Edinb. A, 124 (1994), 517-525.
	
	\bibitem{Galaktionv} V. A. Galaktionv, S. P. Kurdyumov, A. P. Mikhailov and A. A. Samarskill, Unbounded solutions of the Cauchy problem for the parabolic equation $u_t=\nabla(u^\sigma\nabla u)+u^\beta$, Dokl. Akad. Nauk SSSR, 252 (1980), 1362-1364.

	
	\bibitem{Goodman} J. Goodman, Nonlinear asymptotic stability of viscous shock profiles for conservation laws, Arch. Ration. Mech. Anal., 95 (1986), 325-344.
	
	\bibitem{Hashimoto}	I. Hashimoto, Stability of the radially symmetric stationary wave of the Burgers equation with multi-dimensional initial perturbations in exterior domain, Math. Nachr., 293 (2020), 2348-2362.
	

	\bibitem{Howard-2} P. Howard,
	Pointwise Green's function approach to stability for scalar conservation laws, Commun. Pure Appl. Math., 52 (1999), 1295-1313.
	
	\bibitem{Howard-3} P. Howard,
	Pointwise estimates on the Green's function for a scalar linear convection-diffusion equation, J. Differ. Equ., 155 (1999),  327-367.
	
	\bibitem{Howard-1} P. Howard,
	Pointwise estimates and stability for degenerate viscous shock waves, J. Reine Angew. Math., 545 (2002), 19-65.

	
	
	
	\bibitem{Ili'in} A. M. Il'in and O. A. Oleinik, Asymptotic behavior of the solutions of the Cauchy problem for
	certain quasilinear equations for large time, Mat. Sb., 51 (1960), 191-216.	
	
	
	\bibitem{Jin} C. Jin, J. Yin and Y. Ke, Critical extinction and blow-up exponents for fast diffusive polytropic filtration equation with source, Proc. Edinburgh Math. Soc., 52 (2009), 419-444.
	
	\bibitem{Jones}	C. K. R. T. Jones, R. Gardner and T. Kapitula, Stability of traveling waves for non-convex scalar viscous conservation laws, Commun. Pure Appl. Math., 46 (1993), 505-526.	


	\bibitem{Kang-Vasseur-Wang} M.-J. Kang, A. F. Vasseur, Y. Wang, Time-asymptotic stability of composite waves of viscous shock and rarefaction for
	barotropic Navier-Stokes equations, Adv. Math., 419 (2023), 108963.
	
	\bibitem{growing interfaces} M. Kardar, G. Parisi and Y.-C. Zhang, Dynamic scaling of growing interfaces, Phys. Rev. Lett., 56 (1986), 889-892.

	
	\bibitem{same line} S. Kawashima and A. Matsumura, Asymptotic stability of traveling wave solutions of systems for one-dimensional gas motion, Commun. Math. Phys., 101 (1985), 97-127.
	
	\bibitem{nonconvex} S. Kawashima and A. Matsumura, Stability of shock profiles in viscoelasticity with non-convex constitutive relations, Commun. Pure Appl. Math., 47 (1994), 1547-1569.
	
	\bibitem{Kim} Y. J. Kim and A. E. Tzavaras, Diffusive $N$-waves and metastability in the Burgers equation, SIAM J. Math. Anal., 33 (2001), 607-633.
	
	
	
	\bibitem{jet flow} L. Kofman and A. C. Raga, Modeling structures of knots in jet flows with the Burgers equation, Astrophys. J., 390 (1992), 359-364.
	
	\bibitem{Y. Li} Y. Li and J. Wu, Extinction for fast diffusion equations with nonlinear sources, Electron J. Differ. Equ., 2005 (2005), 1-7.
	
	\bibitem{Liu} T.-P. Liu, Nonlinear stability of shock waves for viscous conservation laws, Mem. Am. Math. Soc., 12 (1985), 233-236.
	
	\bibitem{IBVP2} T.-P. Liu, A. Matsumura and K. Nishihara, Behaviors of solutions for the Burgers equation with boundary corresponding to rarefaction waves, SIAM J. Math. Anal., 29
	(1998), 293-308.
	
	\bibitem{IBVP3} T.-P. Liu and K. Nishihara, Asymptotic behavior for scalar viscous conservation laws with boundary effect, J. Differ. Equ., 133 (1997), 296-320.
	
	\bibitem{IBVP1} T.-P. Liu and S.-H. Yu, Propagation of a stationary shock layer in the presence of a boundary, Arch. Ration. Mech. Anal., 139 (1997), 57-82.

	\bibitem{Liu-Zeng-1} T.-P. Liu and Y. Zeng,
	Time-asymptotic behavior of wave propagation around a viscous shock profile, Commun. Math. Phys., 290 (2009),  23-82.
	
	\bibitem{Liu-Zeng-2} T.-P. Liu, Y.  Zeng,
	Shock waves in conservation laws with physical viscosity, Mem. Amer. Math. Soc., 234 (2015), 180 pp.

	\bibitem{Mascia-Zumbrun} C. Mascia and K. Zumbrun,
	Stability of large-amplitude viscous shock profiles of hyperbolic-parabolic systems, Arch. Ration. Mech. Anal., 172 (2004),  93-131.
	
	
	\bibitem{Matsumura-Mei} A. Matsumura and M. Mei, Convergence to travelling fronts of solutions of the p-system with viscosity in the presence of a boundary, Arch. Ration. Mech. Anal.,  146 (1999), 1-22.

\bibitem{Matsumura-Mei2} A. Matsumura and M. Mei, Nonlinear stability of viscous shock profile for a non-convex system of viscoelasticity, Osaka J. Math., 34 (1997) 589--603.
	
	
	\bibitem{Matsumura-Nishihara-1} A. Matsumura, K. Nishihara, On the stability of traveling wave solutions of a
	one-dimensional model system for compressible viscous gas, Japan J. Appl. Math., 2
	(1985) 17-25.
	
	\bibitem{conventional energy method} A. Matsumura and K. Nishihara, Asymptotic stability of traveling waves
	for scalar viscous conservation laws
	with non-convex nonlinearity, Commun. Math. Phys., 165 (1994), 83-96.
	
	\bibitem{Matsumura-Nishihara-Textbook} A. Matsumura, K. Nishihara, Global Solutions to Nonlinear Partial Differential Equations from Compressible Fluid Dynamics (Japanese), Nihon Horun Publisher, 2004.
	
	
	\bibitem{McQuighan} K. McQuighan and C. E. Wayne, An explanation of metastability in the viscous Burgers equation with periodic boundary conditions via a spectral analysis, SIAM J. Appl. Dyn. Syst., 15 (2016), 1916-1961.
	
	\bibitem{M. Mei} M. Mei, Stability of shock profiles for non-convex scalar viscous conservation laws, Math. Models Methods Appl. Sci., 05 (1995), 279-296.

    \bibitem{Mei-Nishihara} M. Mei and K. Nishihara,  Nonlinear stability of travelling waves for one dimensional viscoelastic materials with non-convex nonlinearity, Tokyo J. Math., 20 (1997), 241-264.
	
	\bibitem{Mochizuki} K. Mochizuki and K. Mukai, Existence and nonexistence of global solutions to fast diffusions with source, Methods Appl. Anal., 2 (1995), 92-102.
	
	\bibitem{local existence} T. Nishida, Nonlinear hyperbolic equations and related topics in fluid dynamics, Publ. Math. d'Orsay, 1978.
	
	\bibitem{Nishihara} K. Nishihara, A note on the stability of traveling wave solutions of Burgers' equation, Japan J. Appl. Math., 2 (1985), 27-35.

	
	
	
	\bibitem{Qi} Y. W. Qi, On the equation $u_t=\Delta u ^\alpha+u^\beta$, Proc. R. Soc. Edinb. Sect. A, 123 (1993), 373-390.
	
	\bibitem{Qi2} Y. W. Qi, The critical exponents of degenerate parabolic equations, Sci. China Ser. A, 38 (1995), 1153-1162.
	
	
	
	
	\bibitem{spectural} D. H. Sattinger, On the stability of waves of nonlinear parabolic systems, Adv. Math., 22 (1976), 312-355.
	
	\bibitem{Szepessy} A. Szepessy and Z. Xin, Nonlinear stability of viscous shock waves, Arch. Ration. Mech. Anal.,
	122 (1993), 53-104.
	
	\bibitem{Tian} Y. Tian and C. Mu, Extinction and non-extinction for a p-Laplacian equation with nonlinear source, Nonl. Anal., 69 (2008), 2422-2431.
	
	
	\bibitem{traffic flow} R. M. Velasco and P. Saavedra, A first order model in traffic flow, Physica D, 228 (2007), 153-158.
	
	\bibitem{nonconvex3} H. F. Weinberger, Long-time behavior for a regularized scalar conservation law in the absence of genuine nonlinearity, Ann. Inst. H. Poincar\'{e} Anal. Non Lin\'{e}aire, 07 (1990), 407-425.
	
	\bibitem{Xu-Mei-Qin-Sheng} S. Xu, M. Mei, J.-C. Nave, and W. Sheng, Viscous shocks to Burgers equations with fast diffusion and singularity, preprint, (2023).
	
	\bibitem{Yin1} J. Yin and C. Jin, Critical extinction and blow-up exponents for fast diffusive p-Laplacian with sources, Math. Methods Appl. Sci., 30 (2007), 1147-1167.
	
	\bibitem{Yin2} J. Yin, J. Li and C. Jin, Non-extinction and critical exponent for a polytropic filtration equation, Nonl. Anal., 71 (2009), 347-357.
	
	
	
	
\end{thebibliography}
\end{document}